\documentclass[12pt]{article}
\usepackage{geometry} 
\usepackage[utf8]{inputenc}
\usepackage[T1]{fontenc}
\usepackage[british]{babel}
\usepackage{lmodern}
\usepackage{graphicx}
\usepackage{amsmath}   
\usepackage{amssymb}   
\usepackage{amsfonts}
\usepackage{amsthm}
\usepackage{epstopdf}
\usepackage{mathtools}
\usepackage[pdftex]{color}
\usepackage{xcolor}
\usepackage{bm}
\usepackage{tikz}
\usepackage{url}
\usetikzlibrary{matrix}
\usepackage{multicol}
\usepackage{array}
\usepackage[parfill]{parskip}
\usepackage{fancyhdr}
\usepackage{lipsum}
\usepackage{longtable}
\usepackage{enumerate}
\usepackage{subcaption}
\usepackage{pinlabel}
\usepackage{tikz-cd}
\usepackage[normalem]{ulem}
\usepackage[numbers,sort]{natbib}

\usepackage[pdftex]{color}
\usepackage{xcolor}

\usepackage[parfill]{parskip}
\makeatletter
\def\thm@space@setup{%
  \thm@preskip=\parskip \thm@postskip=0pt
}
\makeatother

\numberwithin{equation}{section}
\newtheorem{theo}{Theorem}[section]
\newtheorem{pro}[theo]{Proposition}
\newtheorem{lemma}[theo]{Lemma}
\newtheorem{cor}[theo]{Corollary}
\newtheorem{re}[theo]{Remark}
\newtheorem{assumption}[theo]{Assumption}

\theoremstyle{plain}

\theoremstyle{definition}
\newtheorem{de}[theo]{Definition}

\newcommand{\mbP}{\mathbb{P}}
\newcommand{\mbC}{\mathbb{C}}

\newcommand{\mbN}{\mathbb{N}}
\newcommand{\mbZ}{\mathbb{Z}}

\newcommand{\Pic}{\mathrm{Pic}}

\newcommand{\cE}{\mathcal E}

\allowdisplaybreaks

\title{\textbf{Dynamical degrees of birational maps\\ from indices of polynomials \\with respect to blow-ups I. \\General theory and 2D examples}}
\author{Jaume Alonso \and Yuri B.\ Suris \and Kangning Wei}
\date{\small Institut für Mathematik, MA 7-1\\ Technische Universität Berlin, Str.\ des 17.\ Juni, 10623 Berlin, Germany\\
E-mail: alonso@math.tu-berlin.de, suris@math.tu-berllin.de, wei@math.tu-berlin.de}

\begin{document}

\maketitle

\begin{abstract}
\noindent
In this paper we address the problem of computing $\deg(f^n)$, the degrees of iterates of a birational map $f:\mbP^N\dasharrow\mbP^N$. For this goal, we develop a method  based on two main ingredients: the factorization of a polynomial under pull-back of $f$, based on local indices of a polynomial associated to blow-ups used to resolve the contraction of hypersurfaces by $f$, and the propagation of these indices along orbits of $f$.\\[0.2cm]
\noindent
For maps admitting algebraically stable modifications $f_X:X\dasharrow X$, where $X$ is a variety obtained from $\mbP^N$ by a finite number of blow-ups, this method leads to an algorithm producing a finite system of recurrent equations relating the degrees and indices of iterated pull-backs of linear polynomials. We illustrate the method by three representative two-dimensional examples. It is actually applicable in any dimension, and we will provide a number of three-dimensional examples as a separate companion paper. 
\end{abstract}

\section{Introduction}

Birational maps belong to the most important objects of the theory of integrable systems. In fact, it was in the context of difference equations and the corresponding birational maps in which one of the most celebrated integrability criteria was formulated, namely the \emph{``singularity confinement''} \cite{Grammaticos_Ramani_1991}. Although it soon became clear that this criterium is neither necessary nor sufficient for integrability, its investigation and application has led to many important discoveries and developments. In particular, the whole area of discrete Painlev\'e equations and their geometric theory has been established \cite{Sakai2001, Grammaticos_Ramani_2004, Kajiwara_Noumi_Yamada_2017, Joshi2019}. A closely related line of research deals with the so called QRT maps and their relatives \cite{QRT1988, QRT1989, Tsuda2004, Duistermaat2010, Carstea_Takenawa_2012, Carstea_Takenawa_2019,Alonso_Suris_Wei_2022}.  An interpretation of the ``singularity confinement'' criterion in the language of projective and algebraic geometry \cite{Hietarinta_Viallet_1997, Hietarinta_Viallet_2000, Takenawa2001, Takenawa_2001b, Mase_Willox_Ramani_Grammaticos_2019}, the introduction of the notion of \emph{``algebraic entropy''}, and the establishment of the vanishing of algebraic entropy as a valid integrability criterium \cite{BellonViallet1999, Viallet_2008} are the further hallmarks on this path. 
 
It should be mentioned that one of the main sources of birational maps in the theory of integrable systems is the problem of \emph{integrable discretization} of continuous time integrable systems. A large number of interesting examples with distinguished geometric properties appeared in the recent years in the context of Kahan-Hirota-Kimura discretization. Such a discretization produces for any quadratic vector field on $\mathbb{R}^n$  a birational map of degree $n$ on $\mathbb{P}^n$, see \cite{HirotaKimura2000, Kahan1993, KimuraHirota2000}. The study of remarkable conservation properties of this class of birational maps, including their integrability under certain circumstances, was initiated in \cite{PPS2009, PPS2011, Celledoni1, Celledoni2}, and led to further findings.

Another line of research leading to the study of the same structures is the multi-dimensional complex dynamics. A representative though, of course, non-exhaustive list of relevant references includes \cite{FornaessSibony2, Diller1996, Bedford_1998, FornaessSibony1, DillerFavre2001, Bedford_2004, Bedford_Kim_2004, Bedford_Diller_2005, Bedford_Diller_2006, Bedford_Kim_2006, McMullen_2007, Bedford_Kim_2009, Bedford_Kim_2010, Bedford_2011, Diller2011, Bedford_Kim_2011, silverman_2012, Blanc_2013, Bedford_Kim_2014, Bedford_2015, Bedford_Diller_Kim_2015, Blanc_Cantat_2016, Deserti_2018, Bell_Diller_Jonsson_2020, Birkett2022}. 

An important quantity associated to a birational map $f:\mbP^N\dasharrow\mbP^N$ is the \textit{dynamical degree} $\lambda(f)$, defined as  $\lambda(f)=\lim_{n\to\infty} (\deg(f^n))^{1/n}$. Its logarithm $h(f)=\log(\lambda(f))$ is called the \textit{algebraic entropy} of $f$, and its vanishing belongs to the most important indicators of integrability for discrete systems. The growth rate of the sequence $\deg(f^n)$ is, thus, crucial for  understanding the dynamics. In the case of dimension $N=2$, one can actually establish a precise correspondence between the growth rate of $\deg(f^n)$ and the underlying geometry of an algebraic surface \cite{DillerFavre2001}. In particular, $\deg(f^n)\sim n^2$ corresponds to an automorphism of an algebraic surface preserving a fibration of elliptic curves.

Our goal in this paper is even more ambitious than the computation of $\lambda(f)$ or the growth rate of $\deg(f^n)$ for a given birational map $f$, namely, it is the \emph{exact computation} of all degrees $\deg(f^n)$. A general method for doing this is available in the literature quoted above, and consists in, first, lifting the map $f$ on $\mbP^N$ to a so called \emph{algebraically stable} modification $f_X$ on some variety $X$ obtained from $\mbP^N$ via a birational morphism $\pi:X\to \mbP^N$ (a composition of blow-up transformations), and, second, studying the induced linear action $f_X^*$ on the Picard group $\Pic(X)$ (the free abelian group generated by divisor classes on $X$). A great advantage of this method is a reduction of a complicated dynamical problem to a problem of linear algebra (studying powers of the matrix of $f_X^*$ relative to a suitable basis). 

The method we propose in the present paper is based on a detailed investigation of the main analytic phenomenon responsible for the fact that $\deg(f^n)\neq (\deg(f))^n$, namely the cancellation of common factors in the components of $f^n$. The most closely related method is that by Viallet \cite{Viallet_2015, Viallet_2019}. The main difference of his approach as opposed to ours is the following: he studies the peculiar factorization patterns in the components of $f^n$, which amounts to the study of iterated pull-backs of the \emph{coordinate hyperplanes}; for doing this, the interplay between the operations of pull-back and push-forward plays a crucial role. For us, on the contrary, the main role is played by the iterated pull-backs of \emph{generic hyperplanes}. Neither do we pay any special attention to the specific behavior of coordinate hyperplanes nor do we consider the push-forward operation.

Our method is based on the use of  local indices of polynomials associated to blowing-ups coming from resolution of the singularities of the birational map. The factorization phenomenon mentioned above is based on the following fact. For a birational map $f$ on $\mbP^N$ (given by its minimal lift $\tilde{f}$ to homogeneous coordinates, a polynomial map on $\mbC^{N+1}$) and for an irreducible homogeneous polynomial $P$ on $\mbC^{N+1}$, the pull-back $f^*P=P\circ \tilde f$ need not be irreducible anymore; rather, it can acquire certain factors of geometric origin:
$$
f^*P=\left(\prod_{i=1}^r K_i^{\nu_i(P)}\right)\widetilde P,
$$
where $\tilde P$ is irreducible and $K_i$ are minimal generating polynomials of the irreducible components of the critical set $\mathcal C(f)$, defined by the equation $\det d\tilde f=0$. Our first contribution is the exact computation of indices $\nu_i(P)$ for an arbitrary homogeneous polynomial $P$ (see Theorem \ref{theorem split off}). It is based on the (local) blow-up maps used to resolve the contraction of the divisors $\{K_i=0\}$ to the respective points $p_i$.

Based on this result, one defines inductively 
$$
P_{n+1}=\left(\prod_{i=1}^r K_i^{-\nu_i(P_n)}\right)(f^*P_n),
$$
then $\deg(f^n)=\deg(P_n)$ for a \emph{generic} homogeneous polynomial $P_0$ of degree 1. Thus, the task of computing $\deg(P_n)$ is reduced to the task of determining the recurrent relations which express $\deg(P_{n+1})$ and $\nu_i(P_{n+1})$ through the data of the polynomial $P_n$. This is achieved by further blow-ups along the orbits of the points $p_i$. If $f$ admits an algebraically stable modification $f_X$, the resulting system of recurrent relations is closed (finite). It is then related to (but does not necessarily coincide with) the analogous system describing the powers of the linear map $f_X^*$ on $\Pic(X)$. 

The paper is organized as follows. In Sections \ref{Sect rational} and \ref{Sect birational} we recall the basic notions and results related to rational and birational maps on $\mbP^N$, based mainly on \cite{Hudson1927, Diller1996, FornaessSibony2, Alberich}.
In Section \ref{sect index} we discuss the problem of factorization of polynomials under pull-backs, introduce our \textit{local index} of a polynomial associated to a blow-up, and prove the main Theorem \ref{theorem split off}. The propagation of the local indices under the birational map is studied in Section \ref{sect propagate}. The discussion here is not exhaustive but covers the most common and useful situations. These results are combined in Section \ref{section first algorithm} to an algorithm for computing the sequence of degrees $\deg(f^n)=\deg(P_n)$, which terminates if $f$ admits an algebraically stable modification $f_X:X\dasharrow X$. The more traditional algorithm based on the study of the action of $f_X^*$ on $\Pic(X)$ is presented in Section \ref{sect method Picard}, and the two methods are compared in Section \ref{sect two methods}. The last three Sections \ref{sect example 1}--\ref{sect linearizable} contain a detailed study of three two-dimensional examples, illustrating various aspects of our method and its realization for concrete maps. 
 
It should be mentioned that our method can be applied to birational maps in all dimensions. In a companion paper, we will provide examples in dimension $N=3$. Of course, in dimension $N=2$ the geometry is more thoroughly understood, which gives several additional insights. For instance, our local indices can be interpreted as linear combinations of intersection numbers between divisors on the surface. In the higher dimensional cases, such an interpretation is not yet available.


\section{Rational maps}
\label{Sect rational}

The general setting for our discussion will be rational maps of $\mathbb P^N$,
\begin{equation}\label{eq: rational map}
f: \mathbb{P}^N\dasharrow\mathbb{P}^N, \quad [x_0:x_1:\ldots:x_N]\mapsto[X_0:X_1:\ldots:X_N],
\end{equation}
where $X_k=X_k(x_0,\ldots,x_N)$, $k=0,\ldots,N$, are homogeneous polynomials of one and the same degree $d$ without a non-trivial common factor. The number $d$ is called the {\em degree} of $f$, denoted by $\deg(f)$. The polynomial map 
\begin{equation}\label{eq: polynomial map}
\tilde f: \mathbb C^{N+1}\to\mathbb C^{N+1}, \quad (x_0,x_1,\ldots,x_N)\mapsto (X_0,X_1,\ldots,X_N)
\end{equation}
will be called a {\em minimal lift} of $f$. It is defined up to a constant factor.
To each rational map we associate:
\begin{itemize}
\item the {\em indeterminacy set} $\mathcal I(f)$ consisting of the points $ [x_0:x_1:\ldots:x_N]\in \mathbb{P}^N$ for which $X_0=X_1=\ldots=X_N=0$;  and
\item the {\em critical set} $\mathcal C(f)$ consisting of the points $[x_0:x_1:\ldots:x_N]\in \mathbb{P}^N$ where $\det d\tilde f=0$; the latter equation of degree $(N+1)(d-1)$ defines a variety of codimension 1.
\end{itemize} 
The set of singular points $\mathcal I(f)$ has codimension at least 2. Otherwise, if $f$ were singular on a codimension 1 variety defined by a polynomial equation $K=0$, then the homogeneous polynomials $X_0,X_1,\ldots,X_N$ that define $f$ would be divisible by $K$, which is a contradiction.

Let $\tilde f^k=\tilde f\circ \tilde f\circ...\circ \tilde f$ ($k$ times). The components of this polynomial map may contain a nontrivial common factor: $\tilde f^k=K\cdot f^{[k]}$, where $K$ is a homogeneous polynomial and components of $f^{[k]}$ have no nontrivial common factor, so that $f^{[k]}$ is a minimal lift of the rational map $f^k$ on $\mathbb{P}^n$. We write

$$
f^{[k]}(x_0,x_1,\ldots,x_N)=\big(X_0^k, X_1^k,\ldots, X_N^k\big).
$$
Thus, $\deg(f^k)$ is the common degree of the homogeneous polynomials $X_0^k,X_1^k,\ldots,X_N^k$ which equals $d^k-\deg K \leq d^k$. An understanding and a control of such a drop of degree is achieved based on the following result.
\begin{pro}\label{deg lowering}
Let $f$ and $g$ be rational maps of $\,\mathbb {P}^N$ with minimal lifts $\tilde f$ and $\tilde g$. Then $\tilde f\circ \tilde g$ is not a minimal lift of $f\circ g$ and $\deg (f\circ g)<\deg f \cdot\deg g$ if and only if there exists a hypersurface $A$ such that $g(A)\subset \mathcal I(f)$. Such a hypersurface, if exists, has to be a subset of $\mathcal C(g)$.
\end{pro}
\begin{proof} If the map $\tilde f\circ \tilde g\ $ is not a minimal lift of $f\circ g$ then there is a homogeneous polynomial $a$ which divides all $N+1$ components of $\tilde f\circ \tilde g$. Let $A=\{a=0\}\subset \mathbb {P}^N$ be the hypersurface defined by $a$. Then $\tilde f$ vanishes at the lift of $g(A)$ to $\mathbb C^{N+1}$, so that $g(A)\subset \mathcal I(f)$. Conversely, for any hypersurface $A$ with this property, $\tilde f\circ \tilde g$ vanishes on the lift of $A$ to $\mathbb C^{N+1}$, and all $N+1$ components of this map are divisible by a minimal defining polynomial of $A$. Finally, we show that a hypersurface with this property has to be a subset of $\mathcal C(g)$: indeed, if $A$ would contain non-critical point of $g$ then the image $g(A)$ would contain a piece of codimension 1 and therefore could not be contained in the variety $\mathcal I(f)$ of codimension 2. 
\end{proof}

A specialization to the most important case, namely when $g$ is an iterate of $f$, leads to the following definition and corollary.

\begin{de}\label{def deg lowering}
A hypersurface $A=\{a=0\}\subset \mathcal C(f)$ is called a  \emph{degree lowering hypersurface} if $f^k(A)\subset \mathcal I(f)$ for some $k\in\mathbb N$.
\end{de}

\begin{cor}
For a rational map $f$ of $\,\mbP^N$, we have $\deg(f^n)=(\deg f)^n$ for all $n\in \mathbb N$ if and only if it does not possess a degree lowering hypersurface. If such a hypersurface exists, then the dynamical degree of the map $f$, defined as
\begin{equation}\label{dyn deg}
\lambda(f):=\lim_{n\to\infty} (\deg(f^n))^{1/n},
\end{equation}
is strictly less than $\deg f$. 
\end{cor}
\begin{proof}
The first statement follows directly from Proposition \ref{deg lowering}. For the second one, set $d_n=\deg(f^n)$. Then $d_{n+m}\leq d_{n}d_{m}$ for any $n,m\in\mathbb N$, and as a consequence of Fekete lemma \cite{Fekete}, the limit in \eqref{dyn deg} always exists, cf. \cite{silverman_2012}. 
If there exists a hypersurface $A$ with $f^k(A)\subset\mathcal I(f)$ for some $k\in \mathbb N$, then by Proposition \ref{deg lowering} we have $d_{k+1}=\deg(f^{k+1})=\deg(f\circ f^k)<(\deg f)^{k+1}$, so that
$
(d_{k+1})^{1/(k+1)}<\deg f.
$
By induction, 
$(d_{(k+1)n})^{1/(k+1)n}\le (d_{k+1})^{1/(k+1)}.$
Therefore, $\lambda(f)\le (d_{k+1})^{1/(k+1)}<\deg f$.
\end{proof}


\section{Birational maps}
\label{Sect birational}

\begin{de}
A rational map $f_+:\mathbb{P}^N \dasharrow \mathbb{P}^N$ is called birational if there is a rational map $f_-:\mathbb{P}^N \dasharrow\mathbb{P}^N$ such that $f_-\circ f_+=id$ and $f_+\circ f_-=id$ away from some subvariety of codimension 1.
\end{de}
This means that
\begin{equation}\label{f+f-}
\tilde{f}_-\circ \tilde{f}_+=K_+\cdot id, \quad  \tilde{f}_+\circ \tilde{f}_-=K_-\cdot id, 
\end{equation}
where $K_+$ and $K_-$ are homogeneous polynomials. The degrees $d_+=\deg(f_+)$ and $d_-=\deg(f_-)$ need not coincide. Notice that $\deg(K_+)=\deg(K_-)=d_+d_--1$.

\begin{pro} \label{prop birational gen}\quad {\rm\cite[Proposition 3.3 and Corollary 3.6]{Diller1996}}

\begin{itemize}
\item[{\rm a)}] The image of $\mathcal C(f_+)\setminus \mathcal I(f_+)$ under the map $f_+$ belongs to $\mathcal I(f_-)$ (in particular, it is of codimension $\ge 2$).
\item[{\rm b)}] There holds $\mathcal I(f_+)\subset \mathcal C(f_+)$. Moreover, every irreducible component of $\mathcal C(f_+)$ contracted by $f_+$ to a point contains a codimension 2 subvariety from $\mathcal I(f_+)$. 
\item[{\rm c)}] The restriction $f_+:\mathbb{P}^N\setminus \mathcal{C}(f_+)\to \mathbb{P}^N\setminus\mathcal{C}(f_-)$ is biregular. 
\item[{\rm d)}] The polynomial $K_+$ is a (in general non-minimal) defining polynomial of $\mathcal C(f_+)$, so that 
$\mathcal C(f_+)=\{K_+=0\}$. 
\end{itemize}
\end{pro}
\begin{proof} \quad These statements are well known and can be found in the classical book  \cite{Hudson1927} and in a more modern presentation (for 2D only) in \cite{Diller1996, FornaessSibony2,Alberich}.
\end{proof}

Thus, an important feature of a birational map $f_+$ (which does not hold for general rational maps) is that it contracts its critical set to a variety of codimension $\ge 2$. Notice that neither $\det(d{\tilde f}_+)$ nor $K_+$ must be a minimal defining polynomial of the critical set $\mathcal C(f_+)$. The degree $\det(d{\tilde f}_+)$ equals $(N+1)(d_+-1)$, while $\deg(K_+)=d_+d_--1$. In any case, if we denote by $K_1,\ldots,K_r$ the minimal defining polynomials of the irreducible components of the critical set $\mathcal{C}(f_+)$, then 
$$
\det({\tilde f}_+)=\prod_{i=1}^r K_i^{\alpha_i}, \quad K_+=\prod_{i=1}^r K_i^{\beta_i}, \quad \alpha_i,\beta_i\in\mathbb N.
$$

In what follows, we will denote a birational map $f_+$ simply by $f$, and we will call $f_-$ the \emph{birational inverse} of $f$, and write $f_-=f^{-1}$. Let $A\subset\mathcal C(f)$ be a degree lowering hypersurface. According to Proposition \ref{prop birational gen}, part a), the image $f(A)\subset \mathcal I(f^{-1})$ is a variety of codimension $\ge 2$. For $N=2$, this means $f(A)$ is a point, but for $N\ge 3$, $f(A)$ can be more complicated. If the dimension of $f(A)$ is $\ge 1$, then the images of $f(A)$ under further iterates of $f$ become more and more complicated algebraic varieties (i.e.\ of higher and higher degree). Therefore it becomes increasingly difficult to arrange (by adjusting parameters of the map $f$, say) that some $f^k(A)$ is a subset of $\mathcal I(f)$. One can hope to be able to arrange this only if $f(A)$ is a point, so that some iterate of this point under $f^{k-1}$ has to land at a point belonging to $\mathcal I(f)$. Therefore, we make the following assumption: 

\begin{assumption} \label{assumption}
 All degree lowering components of $\mathcal C(f)$ are contracted by $f$ to points. 
 \end{assumption}

We stress that this condition is satisfied in all our examples. Of course, there could exist interesting examples where this assumption is violated.


\section{Control of factorization under pull-back: \\ local indices}
\label{sect index}

The question whether the $N+1$ components of $\tilde f^k$ have a common factor, so that $\deg(f^k)<(\deg(f))^k$, reduces to the problem of factorization and common factors of the successive pull-backs of the coordinate monomials $x_0,\ldots,x_N$, where the pull-back of a homogeneous polynomial  $P=P(x_0,\ldots,x_N)$ is defined as
\begin{equation}\label{pull back}
f^*P=P\circ \tilde f.
\end{equation}

The role of components of the critical set $\mathcal C(f)$ contracted to points is clarified by the following statement. 

\begin{pro}\label{lemma prefactor}
Let $f:\mbP^N\dasharrow\mbP^N$ be a birational map, let $\{K=0\}$ be an irreducible component of $\mathcal C(f)$ (with $K$ being its minimal defining polynomial), such that $f$ contracts $\{K=0\}\setminus \mathcal I(f)$ to a point $p\in\mathcal I(f^{-1})$. Then, for any homogeneous polynomial $P$ such that $\{P=0\}$ passes through the point $p$, its pull-back $f^*P$ is divisible by $K$. Moreover, if the multiplicity of the point $p$ on the hypersurface $\{P=0\}$ is $m\ge 1$, then $f^*P$ is divisible by $K^m$.
\end{pro}
Informally speaking, each branch of a hypersurface $\{P=0\}$ through the singularity $p\in\mathcal I(f^{-1})$ causes an extra factor $K$ in the pull-back of the polynomial $P$ under $f$ to appear.

\begin{proof}
We start with the case $m=1$. We have: $f^*P(x_0,\ldots,x_N)=P(\tilde f(x_0,\ldots,x_N))$. By hypothesis, whenever $[x_0:\ldots:x_N]\in \mbP^N$ satisfies  $K(x_0,\ldots,x_N)=0$ (and does not belong to $\mathcal I(f)$), its image $\tilde f(x_0,\ldots,x_N)=\tilde p$ is a lift of $p$, so that  $P(\tilde f(x_0,\ldots,x_N))=P(\tilde p)=0$. Thus, $P\circ \tilde f=f^*P$ is divisible by $K$.

In the general case $m\ge 1$, we assume, without loss of generality, that $p=[1:p_1:\ldots:p_N]$. Set $\xi_i=x_i/x_0$ for $i=1,\ldots,N$. We have the Taylor expansion at $p$:
\begin{equation}\label{Taylor}
P(1,\xi_1,\ldots,\xi_N)=\sum_{i_1+\ldots+i_N\ge m} c_{i_1\ldots i_N}(\xi_1-p_1)^{i_1}\cdots (\xi_N-p_N)^{i_N}
\end{equation}
(the sum is finite). We denote $\tilde f(x_0,\ldots,x_N)=(X_0,\ldots,X_N)$. By hypothesis, polynomials
$X_i-p_iX_0$ with $i=1,\ldots,N$ vanish as soon as $K(x_0,\ldots,x_N)=0$, thus, they are divisible by $K(x_0,\ldots,x_N)=0$. This implies that (every term of) the polynomial
$$
P(X_0,\ldots,X_N)=\sum_{i_1+\ldots+i_N\ge m} c_{i_1\ldots i_N}(X_1-p_1X_0)^{i_1}\cdots (X_N-p_NX_0)^{i_N}X_0^{\deg(P)-i_1-\ldots-i_N}
$$
is divisible by $K^m$.
\end{proof}

One can prove a similar statement for any irreducible component $\{K=0\}$ of $\mathcal C(f)$ whose $f$-image is a variety of dimension between 1 and $N-2$, namely that $f^*P$ acquires a factor which is a certain power of $K$  if $\{P=0\}$ contains the $f$-image of $\{K=0\}$. However, due to Assumption \ref{assumption} we do not consider such components any further. 

\begin{de}\label{factorization under pullback}
Given a homogeneous polynomial $P$,  its pull-back under $f$ is given by
\begin{equation}\label{def Ptilde}
f^*P=K_1^{\nu_1}\cdots K_r^{\nu_r}\widetilde P,
\end{equation}
where $K_1,\ldots,K_r$ are the minimal defining polynomials of the components of $\mathcal C(f)$ contracted by $f$ to points, and $\widetilde P$ is not divisible by any of $K_1,\ldots,K_r$. We call the polynomial $\widetilde P$ the \emph{proper pull-back of $P$ under $f$}.
\end{de}

It is important to notice that Proposition \ref{lemma prefactor} gives actually only a lower bound for the exponents $\nu_i$ in \eqref{def Ptilde}, namely 
\begin{equation}
\nu_i\ge \; \textrm{multiplicity of the point} \; p_i=f(\{K_i=0\}) \; \textrm{on}\;  \{P=0\}. 
\end{equation}
To determine these exponents precisely, we use the concept of blow-up: given an $N$-dimensional variety $Y$, the \emph{blow-up} at a point $p\in Y$ is a variety $X$ and a birational map $\pi:X\rightarrow Y$ such that $E:=\pi^{-1}(p)\cong \mathbb{P}^{N-1}$, and $\pi:X\setminus E \rightarrow Y\setminus\{p\}$ is biregular. We write $X=Bl_p(Y)$. The hypersurface $E\subset X$ is called the \emph{exceptional divisor} of the blow-up. In particular, if $Y=\mbP^N$  then $\pi$ is the famous \emph{sigma-process}, cf.\ \cite{Shafarevich2013}, conveniently written in homogeneous coordinates $[x_0:x_1:\ldots:x_N]$ as follows. Suppose (without loss of generality) that $p$ lies in the affine space $\{x_0\neq 0\}$, so that $p=[1:p_1:\ldots:p_N]$. Then, in one of the affine coordinate patches on $Bl_p(\mbP^N)$ around $E$, with coordinates $(u_1,\ldots,u_N)$, the projection $\pi$ is given by
\begin{equation}\label{sigma}
\pi: \; x_0=1,\quad x_1=p_1+u_1,\quad x_2=p_2+u_1u_2,\; \ldots \; x_N=p_N+u_1u_N.
\end{equation}
The exceptional divisor $E$ in $Bl_p(\mbP^N)$ is isomorphic to $\mbP^{N-1}$ given in affine coordinates by $\{u_1=0\}$. For any hypersurface $H$ in $Y$, the preimage $\pi^{-1}(H)$ is called the \emph{total transform} of $H$ in the blow-up variety $X$, while $\overline{\pi^{-1}(H\setminus\{p\})}$ is called the \emph{proper transform} of $H$. 

Blow-ups can be used to \emph{resolve} singularities of birational maps. Generically, if $f$ contracts $\{K=0\}$ to a point $p$,  then the lift of $f$ to $Bl_{p}(\mbP^N)$ maps  (the proper transform of) $\{K=0\}$ to $E$ birationally. In such a case, we say that the singularity $p$ of $f^{-1}$ is \emph{simple}. In more complicated cases, the lift of $f$ to $Bl_p(\mbP^N)$ could still contract $\{K=0\}$ to a subvariety of $E$ of codimension $>1$. Then a further blow-up of this image is necessary. The expressions for a blow-up of a subvariety are similar to \eqref{sigma}, see \cite{Shafarevich2013}, and can always be arranged to be polynomial in suitable affine coordinates. The successive blow-ups are performed until the lift of $f$ maps $\{K=0\}$ birationally  to a subvariety of codimension 1, in which case we say that the singularity $p\in\mathcal I(f^{-1})$ has been resolved. 

Assuming that the process of successive blow-ups terminates, we obtain, for each irreducible component $\{K=0\}$ of $\mathcal C(f)$ contracted to a point $p$, a top level blow-up at $p$, whose exceptional divisor is a birational image of $\{K_i=0\}$.
Assume without loss of generality that $p$ lies in the affine space $x_0\neq 0$. Then the above mentioned top level blow-up at $p$ can be described by a local change of coordinate  
$$
\phi: (u_1,u_2,\ldots,u_N)\to [1:x_1:\ldots:x_N],
$$ 
in which the exceptional divisor $E$ is given by $\{u_1=0\}$.  We may assume that all components $x_k(u_1,\ldots,u_N)$ are polynomials such that $[1:x_1:\ldots:x_N]|_{u_1=0}=p$. Moreover $\phi$ is birational, so the functions $u_i=u_i(x_1,\ldots,x_N)$ are rational. They can (and should) be naturally considered also as rational functions in homogeneous coordinates, $u_i=u_i(x_0,x_1,\ldots,x_N)$. Since $f$ establishes a birational map of $\{K=0\}$ to $E$, we can assume that
$$
f^*u_1=K\cdot \widetilde{u}_1, \quad f^*u_2=\widetilde{u}_2,\;\ldots\; f^*u_N=\widetilde{u}_N,
$$
where neither of the rational functions $\widetilde{u}_i$ is constant on $\{K=0\}$ (in particular, neither of $\widetilde{u}_i$ vanishes identically on $\{K=0\}$).
 
\begin{de}\label{def index}
The \emph{local index} $\nu_{\phi}(P)$ of a homogeneous polynomial $P=P(x_0,x_1,\ldots,x_N)$ with respect to the blow-up map $\phi$ is defined by
\begin{equation}\label{local-indice}
    P\circ \phi=P(1,x_1(u),\ldots,x_N(u))=u_1^{\nu_{\phi}(P)}\bar{P}(u_1,\ldots,u_N),
\end{equation}
where $\bar P$ is a polynomial such that $\bar{P}(0,u_2,\ldots,u_N)\not\equiv 0$. 
\end{de}

Thus, $\nu_\phi(P)$ is the highest degree of $u_1$ which divides the polynomial $P\circ\phi$.

\begin{re} \label{remark multiplicity}
If $p$ is a simple singularity of $f^{-1}$, and $\phi$ is the standard sigma-process \eqref{sigma}, then $\nu_\phi(P)$ coincides with the multiplicity of point $p$ on the hypersurface $\{P=0\}$.
\end{re}

This follows directly from formula \eqref{sigma}.

\begin{theo} \label{theorem split off}
For a birational map $f$ of $\,\mbP^N$, let $\{K=0\}$ be a component of ${\mathcal C}(f)$ mapped by $f$ to a point $p$. Let $\phi:X\to\mbP^N$ be a blow-up map over $p$ such that the lift of $f$ to $X$ maps the proper image of $\{K=0\}$ in $X$ birationally to the exceptional divisor $\{u_1=0\}$. Then, for any homogeneous polynomial $P$, the pull-back $f^*P$ is divisible by $K^{\nu_\phi(P)}$, but not by any higher power of $K$.
\end{theo}
\begin{proof} 
We have:
$$
f^*P(x)=(f^*u_1)^{\nu_\phi(P)}\bar{P}(f^*u_1,f^*u_2,\ldots,f^*u_N)=K^{\nu_\phi(P)}\widetilde{u}_1^{\nu_\phi(P)}\bar{P}(K\widetilde u_1,\widetilde u_2,\ldots,\widetilde u_N).
$$
This function, by definition, is a polynomial on $x_0,\ldots,x_N$. The last formula shows that this polynomial is divisible by $K^{\nu_\phi(P)}$ and by no higher degree of $K$, since $\bar{P}(0,u_2,\ldots,u_N)\not\equiv 0$.
\end{proof}

This theorem will be illustrated by concrete computations of the involved objects and changes of variables in the examples below, compare Propositions \ref{lemma prefactor z y-z}, \ref{lemma Example 3 pullback}.


\section{Propagation of local indices}
\label{sect propagate}

Now that we have control over the possible factorizations of polynomials under the pull-back, we apply this to the problem of the exact computation of degrees of iterates of $f$.

\begin{de}\label{inductive factorization}
Given a homogeneous polynomial $P=P_0$, we define  inductively the polynomials $P_n$ for $n\geq 1$ by 
\begin{equation}\label{def Pn}
f^*P_n=K_1^{\nu_1(n)}\cdots K_r^{\nu_r(n)}P_{n+1},
\end{equation}
where $K_1,\ldots,K_r$ are the minimal defining polynomials of the components of $\mathcal C(f)$ contracted by $f$ to points, and $P_{n+1}$ is not divisible by any of $K_1,\ldots,K_r$, so that $P_{n+1}$ is the proper pull-back of $P_n$ under $f$. 
\end{de}

\begin{pro}
For a generic hyperplane $P=P_0$, we have $\deg(f^n)=\deg(P_n)$.
\end{pro}
\begin{proof}
Recall that $\deg(f^n)$ can be determined as $\deg(f^{[n]})$, where $f^{[n]}$ is defined by the relation $(\tilde f)^n=K\cdot f^{[n]}$, where $\tilde f$ is the minimal lift of the map $f$, and $K$ is the common factor appearing in all $N+1$ components of $(\tilde f)^n$. The $N+1$ components are nothing else but the $n$-th pullback polynomials of the coordinate functions $x_i$, $i=0,\ldots,N$. Clearly, $K$ can be also defined as the maximal factor that splits off the $n$-th iterated pullback $(f^*)^n P_0$ for a generic linear form $P_0=P_0(x_0,\ldots,x_N)$. This proves the statement. 
\end{proof}

Thus, to determine $\deg(f^n)$, we have to compute, for a generic linear polynomial $P_0$, the precise sequences of exponents $\nu_i(n)$ for all components $\{K_i=0\}\subset\mathcal C(f)$ contracted by $f$ to points. 
In the present section, we relate the sequence of exponents $\nu(n)$ for an irreducible component $\{K=0\}\subset\mathcal C(f)$ contracted by $f$ to a point $p$,  to certain local indices defined at the points of the orbit $f^m(p)$. We show how the local indices can be propagated along such an orbit. We do not provide a complete treatment, but rather discuss several representative cases which allow us to give a full account in our examples.

\begin{itemize}
\item \textbf{Case 1:} If $f^{m-1}(p)\notin\mathcal C(f)$, the map $f$ is biregular in a neighborhood of $f^{m-1}(p)$, we can transfer the tower of blow-ups from $f^{m-1}(p)$ to $f^m(p)$. This is done as follows: let $\phi_{m-1}$ be the blow-up at $f^{m-1}(p)$; it is a composition of elementary blow-ups along subvarieties, and the exceptional divisor $E_{m-1}$ of the top-level elementary blow-up is a birational image of the (proper transform of) $\{K=0\}$ under (the lift of) $f^{m-1}$. The lift of $f$ still contracts $E_{m-1}$ to the point $f^m(p)$. To resolve this singularity, we perform the algorithm described in the previous section. Namely, we perform the sequence of  blow ups at $f^m(p)$ until we arrive at an exceptional divisor $E_m$ of the top-level elementary blow-up which is a birational image of $E_{m-1}$. Actually, due to the biregular nature of $f$ at the neighborhood of $f^{m-1}(p)$, the structure of the sequence of blow-ups at $f^{m}(p)$ (i.e., the number of elementary blow ups and the dimensions of the subvarieties we blow up) is isomorphic to the analogous structure at $f^{m-1}(p)$. The composition of the so constructed elementary blow-ups will be $\phi_m$, the blow-up at $f^m(p)$. Since all $K_i\neq 0$ at $f^{m-1}(p)$, we derive from \eqref{def Pn}:
\begin{equation}\label{propagation regular}
\nu_{\phi_{m-1}}(P_{n+1})=\nu_{\phi_m}(P_n).
\end{equation}

\item \textbf{Case 2:} Now consider the case $f^{m-1}(p)\in\mathcal C(f)\setminus\mathcal I(f)$. As in case 1, let $\phi_{m-1}$ be the blow-up at $f^{m-1}(p)$ such that the exceptional divisor $E_{m-1}$ of the top-level elementary blow-up is a birational image of the (proper transform of) $\{K_i=0\}$ under (the lift of) $f^{m-1}$. As in case 1, the lift of $f$ still contracts $E_{m-1}$ to the point $f^m(p)$, and we resolve this singularity by blowing up at the point $f^m(p)$ until  $E_{m-1}$ is birationally mapped to the top level exceptional divisor $E_m$ at $f^m(p)$. However, this time the exact structure of the sequence of blow-ups (the number of elementary blow ups and the dimensions of the subvarieties we blow up) at $f^{m}(p)$ depends on the map $f$ and is not isomorphic to the analogous structure at $f^{m-1}(p)$. In Section \ref{sect example 2} we consider an example where the point $p_1\in \mathcal C(f)\setminus\mathcal I(f)$ falls into this case, and where we need two blow-ups at $p_1$ and three blow-ups at $p_2=f(p_1)$.

To determine the index of $P_{n+1}$ at $E_{m-1}$, we have to consider the relation \eqref{def Pn} upon the blow-up at the point $f^{m-1}(p)$, 
\begin{equation}\label{phi m-1}
\phi_{m-1}: (u_1,\ldots,u_N)\mapsto[x_0:x_1:\ldots:x_N],
\end{equation}
in which the exceptional divisor $E_{m-1}$ is given by $\{u_1=0\}$. We have also a blow up chart
$$
\phi_{m}: (\tilde u_1,\ldots,\tilde u_N)\mapsto[x_0:x_1:\ldots:x_N],
$$ 
in which the exceptional divisor $E_m$ is given by $\{\tilde u_1=0\}$. The map between $\{u_1=0\}$ and $\{\tilde u_1=0\}$ established by $f$ is birational.
Since $f^{m-1}(p)\not\in\mathcal I(f)$, the expression $\tilde f\circ \phi_{m-1}$ does not vanish identically at $u_1=0$, and it is convenient to consider the map $\tilde f\circ \phi_{m-1}$ as an affine chart for the blow-up at $f^m(p)$ in which $E_m$ is parametrized as $\{u_1=0\}$. Thus, we can assume that 
\begin{equation}
\tilde f\circ \phi_{m-1}=\phi_m.
\end{equation}
This allows us to evaluate the left-hand side of \eqref{def Pn} at $\phi_{m-1}$:
$$
f^*(P_n)\circ \phi_{m-1}=P_n(\tilde f\circ\phi_{m-1})=P_n(\phi_m).
$$
Further, some components $\{K_i=0\}$ of $\mathcal C(f)$ pass through the point $f^{m-1}(p)$. For such components, we have
\begin{equation}\label{K on blowup}
K_i\circ\phi_{m-1}\sim u_1^{\nu_{\phi_{m-1}}(K_i)}.
\end{equation}
Plugging all this into \eqref{def Pn}, we find:
\begin{equation}\label{propagation critical}
\nu_{\phi_{m-1}}(P_{n+1})=\nu_{\phi_m}(P_n)-\nu_1(n)\nu_{\phi_{m-1}}(K_1)-\ldots-\nu_r(n)\nu_{\phi_{m-1}}(K_r).
\end{equation}
This formula comes to replace \eqref{propagation regular} in the present case.

\item \textbf{Case 3:} Finally, we consider the case $f^{m-1}(p)\in \mathcal I(f)$. We can have different possibilities concerning the fate of the top-level exceptional divisor $E_{m-1}$.

\begin{enumerate}[(a)]
\item It can happen that $E_{m-1}$ is mapped birationally to some irreducible component of $\mathcal C(f^{-1})$. This is related to the notion of ``singularity confinement'', cf. \cite{BellonViallet1999}. To determine the index of $P_{n+1}$ at $E_{m-1}$, we have to consider the relation \eqref{def Pn} upon the blow-up \eqref{phi m-1} at the point $f^{m-1}(p)$,  in which the exceptional divisor $E_{m-1}$ is given by $\{u_1=0\}$. Since $f^{m-1}(p)\in\mathcal I(f)$, all $N+1$ components of $\tilde f\circ \phi_{m-1}$ vanish at $u_1=0$, therefore, every component is divisible by some power of $u_1$. Let $s\in\mbN$ be the the power of the greatest common divisor of the form $u_1^s$ of the $N+1$ components of $\tilde f\circ \phi_{m-1}$. Then $u_1^{-s}(\tilde f\circ \phi_{m-1})|_{u_1=0}$ is a parametrization of the corresponding component of  $\mathcal C(f^{-1})$.  Since generically the polynomial $P_n$ does not vanish identically on this component of  $\mathcal C(f^{-1})$, we have:
$$
(f^*P_n)\circ \phi_{m-1}\sim u_1^{s\deg(P_n)}.
$$ 
Further, some components $\{K_i=0\}$ of $\mathcal C(f)$ pass through the point $f^{m-1}(p)$. For such components, we have formula \eqref{K on blowup}. Plugging all this into \eqref{def Pn}, we find:
\begin{equation}\label{propagation inverse critical}
\nu_{\phi_{m-1}}(P_{n+1})=s\deg(P_n)-\nu_1(n)\nu_{\phi_{m-1}}(K_1)-\ldots-\nu_r(n)\nu_{\phi_{m-1}}(K_r).
\end{equation}

\item If $E_{m-1}$ is not mapped birationally to some component of $\mathcal C(f^{-1})$, the fate of the exceptional divisor $E_{m-1}$ under $f$ could be much more complicated. Instead of trying to classify all possible situations, we shall consider one particular case which is illustrated by the example of Section \ref{sect linearizable}. In this case, $E_{m-1}$ is mapped birationally to a top-level exceptional divisor of a blow-up at some point from $\mathcal I(f)$. This point can be denoted by $f^m(p)\in \mathcal I(f)$ (which is of course an abuse of notation), and the image by $f$ of $E_{m-1}$ can be denoted by $E_m$. To determine the index of $P_{n+1}$ at $E_{m-1}$, we have to consider the relation \eqref{def Pn} upon the blow-up \eqref{phi m-1} at the point $f^{m-1}(p)$,
in which the exceptional divisor $E_{m-1}$ is given by $\{u_1=0\}$. Since $f^{m-1}(p)\in\mathcal I(f)$, all $N+1$ components of $\tilde f\circ \phi_{m-1}$ vanish at $u_1=0$, therefore, every component is divisible by some power of $u_1$. Let $s\in\mbN$ be the the power of the greatest common divisor of the form $u_1^s$ of the $N+1$ components of $\tilde f\circ \phi_{m-1}$. Then we can use the map 
$$
\phi_m:=u_1^{-s}(\tilde f\circ \phi_{m-1})
$$ 
as an affine chart for the blow-up at $f^m(p)$ in which $E_m$ is parametrized as $\{u_1=0\}$.  We have:
$$
(f^*P_n)\circ \phi_{m-1}\sim u_1^{s\deg(P_n)}u_1^{\nu_{\phi_m}(P_n)}.
$$ 
Further, some components $\{K_i=0\}$ of $\mathcal C(f)$ pass through the point $f^{m-1}(p)$. For such components, we have \eqref{K on blowup}, as before. Plugging all this into \eqref{def Pn}, we find:
\begin{equation}\label{propagation singular}
\nu_{\phi_{m-1}}(P_{n+1})=s\deg(P_n)+\nu_{\phi_m}(P_n)-\nu_1(n)\nu_{\phi_{m-1}}(K_1)-\ldots-\nu_r(n)\nu_{\phi_{m-1}}(K_r).
\end{equation}
This formula can be considered as a mixture of \eqref{propagation critical} and \eqref{propagation inverse critical}.
\end{enumerate}
\end{itemize}


\section{Algorithm for computing $\deg(f^n)$}
\label{section first algorithm}

Based on the results of two previous sections, we now propose an algorithm for computing the degrees of the iterates of a birational map $f:\mathbb P^N\dasharrow \mathbb P^N$ satisfying Assumption \ref{assumption}: 
\begin{itemize}
\item \textbf{Step 1:} For all irreducible components $\{K_i=0\}$ of $\mathcal C(f)$ contracted by $f$ to points $p_i$, perform the blow-up of $\mbP^N$ so that  the lift of $f$ to the blow-up variety maps the proper transform of $\{K_i=0\}$ birationally to the corresponding exceptional divisor of a (local) blow-up map $\phi_i$, as described in Section \ref{sect index}.
\item \textbf{Step 2:} Propagate the blow-ups along the orbits of the points $p_i$, as described in Section \ref{sect propagate}.
\item \textbf{Step 3:} For a generic homogeneous polynomial $P_0$, define its iterated proper pullbacks $P_n$ by Definition \ref{inductive factorization}. Compute the indices of $P_n$ along the orbits of $p_i$ according to formulas for the propagation of indices, see \eqref{propagation regular}, \eqref{propagation critical}, \eqref{propagation inverse critical}, and \eqref{propagation singular}. Compute the degrees of $P_n$ according to the obvious consequence of \eqref{def Pn}:
\begin{equation}
\deg(P_{n+1})=d_+\deg(P_n)-\nu_1(n)\deg(K_1)-\ldots-\nu_r(n)\deg(K_r).
\end{equation}
Here, according to Theorem \ref{theorem split off}, the numbers $\nu_i(n)$ are local indices of the polynomial $P_n$ at the blow ups resolving the contraction of components $\{K_i=0\}$ to corresponding points.  
\end{itemize}
 This algorithm gives a system of recurrence relations for the desired degrees $\deg(f^n)=\deg(P_n)$. However, in general this algorithm does not terminate, so that the system of recurrent relations is infinite. There are two scenarios leading to a finite (closed) system of recurrent relations.
 \begin{itemize}
 \item The orbit $\{f^m(p)\}$ does not land at $\mathcal I(f)$, so that the component $\{K=0\}\subset\mathcal C(f)$ is not degree lowering. In this situation we are in the case 1) for all $m\ge 1$. For a generic polynomial $P_0$, the varieties $\{P_n=0\}$ for the proper pull-backs $P_n$ will avoid the sequence of points $\{f^m(p)\}$, and the recurrent formula \eqref{propagation regular} implies $\nu_{\phi_m}(P_n)=0$.
 \item The orbit of point $p$ lands after a finite number of iterates with the case 3 (a), so that the singularity $p\in\mathcal I(f^{-1})$ is confined in the sense of  \cite{BellonViallet1999}. The right-hand side of the corresponding recurrent formula \eqref{propagation inverse critical} does not contain $\nu_{\phi_m}(P_n)$, which terminates the corresponding part of the recurrent system.
 \end{itemize}
 
If all points $p_i=f(\{K_i=0\})$ fall into one of these two cases, we consider the variety $X$ obtained upon blowing up all orbits of $p_i$ corresponding to degree lowering hypersurfaces. Denote by $f_X$ the lift of $f$ to $X$. Then $f_X$ has the following fundamental property.
 
\begin{de}
A rational map $f_X:X\dasharrow X$ of an algebraic variety $X$ is called \emph{algebraically stable} if it does not possess degree lowering hypersurfaces.
\end{de}

Summarizing, we developed an algorithm for computing degrees of iterates for birational maps which admit algebraically stable modifications. Moreover, there exist examples when our algorithm of blowing up does \emph{not} terminate but still allows us to compute the full sequence of degrees $\deg(P_n)$, see the example in Section \ref{sect linearizable}.


\section{Alternative method for computing $\deg(f^n)$:\\ Picard group}
\label{sect method Picard}

In the dimension $N=2$, Diller and Favre \cite{DillerFavre2001} proved the following result:

\begin{theo}\label{th Diller-Favre}
Let $f$ be a birational map on $\mbP^2$, then there exists a proper modification $\pi: X\rightarrow \mbP^2$, i.e. a composition of finitely many sigma-processes, such that $f$ is lifted to an algebraically stable map $f_X$ on $X$.     
\end{theo}

The following algorithm for finding an algebraically stable modification for a birational map $f$ on $\mbP^2$ was proposed in \cite{DillerFavre2001}; a proof that it terminates after a finite number of steps was given in \cite{Birkett2022}. To formulate the algorithm, one needs the notion of a minimal destablizing orbit.
\begin{de}
An orbit $p,f(p),...,f^{n}(p)$ is called \emph{destablizing} if $f^{-1}(p)$ is a curve and $f^{n}(p)\in \mathcal I(f)$. An orbit is called \emph{minimal destablizing} if, in addition, the points $p_j=f^j(p)$ for $0\le j\leq n$ are distinct, and we have $p_j\notin \mathcal I(f^{-1})$ for $0<j\leq n$ and  $p_j\notin \mathcal I(f)$ for  $0\leq j<n$.
\end{de}
Note that $f^{-1}(p)$ in this definition is exactly a degree lowering curve in the sense of Definition \ref{def deg lowering}. The following simple algorithm produces an algebraically stable map:
\begin{enumerate}
    \item Set $m=0$, $X_0=\mbP^2$ and $f_0=f:X_0\dasharrow X_0$.
    \item Induction in $m$: if $f_m:X_m\dasharrow X_m$ is not algebraically stable, pick a minimal destablizing orbit $p_0,\ldots,p_n$ (such an orbit always    exists), and blow up each of the points $p_j$ to produce the surface $X_{m+1}$; denote by $f_{m+1}:X_{m+1}\dasharrow X_{m+1}$ the lift of $f_m$ to $X_{m+1}$. Repeat until the resulting map is algebraically stable.
\end{enumerate}
We note that the blow-ups in this algorithm are performed in the different order from blow-ups in our algorithm in Section \ref{section first algorithm}. 

In the case when a birational map $f:\mbP^N\dasharrow\mbP^N$ can be lifted to an algebraically stable map $f_X:X\dasharrow X$, there exists another method for computing the degrees, which relates the degrees $\deg(f^n)$ to the induced linear action $f_X^*$ on the Picard group $\Pic(X)$ of the variety $X$. In particular, this method is always applicable in dimension $N=2$ (according to Theorem \ref{th Diller-Favre}).  Applications of this method in dimension $N=2$ are already abundant in the literature, see \cite{DillerFavre2001, Takenawa2001, Takenawa_2001b, Bedford_Diller_2005, Bedford_Diller_2006, Bedford_Kim_2006, McMullen_2007, Bedford_Kim_2010, Bedford_Kim_2011}. For applications in dimension $N>2$, we refer to \cite{Bedford_Kim_2004, Bedford_Kim_2014, Carstea_Takenawa_2019}.

Recall that for a variety $X$ one can define its Picard group $\Pic(X)$ as the free abelian group generated by the divisor classes on $X$. For a variety obtained from blowing ups on $\mbP^N$, one has an explicit description of $\Pic(X)$. If $X$ is obtained from $\mbP^N$ by a sequence of $n$ point blow-ups, and $\mathcal{E}_1,\ldots,\mathcal{E}_n$ are the total transforms of the centers of the blow-ups in $X$, then 
$\Pic(X)=\mathbb{Z}H\bigoplus\mathbb{Z}\mathcal{E}_1\bigoplus\ldots\bigoplus\mathbb{Z}\mathcal{E}_n$, where $H$ is the divisor class of the proper transform of a generic hyperplane in $\mbP^N$. The divisor $\mathcal{E}_i$ is called the exceptional divisor associated with the $i$-th blow-up. 

An important equivalent condition of algebraic stability is given in terms of ${\rm Pic}(X)$. Recall that every birational map $f_X$ on $X$ induces a linear action $f_X^*$ on $\Pic(X)$ via pull-back. 

\begin{pro}\label{AS-def}
A rational map $f_X:X\dasharrow X$ is algebraically stable if and only if $(f_X^n)^*=(f_X^*)^n$ for all $n\in\mathbb N$.
\end{pro}
We refer to \cite{DillerFavre2001} for a proof in the case of dimension $N=2$, where the arguments can be easily generalized to arbitrary dimensions.

Given a birational map $f$ on $\mbP^N$, lifted to an algebraically stable map $f_X$ on a blow-up variety $X$,
its degree $\deg(f)$ equals the coefficient of $H$ in the expansion of the divisor class $f_X^*H$ in the basis $H,\mathcal{E}_1,...,\mathcal{E}_n$, or, in other words, the $(H,H)$ entry of the matrix corresponding to the action $f_X^*$ in this basis. According  to Proposition \ref{AS-def}, we have:

\begin{pro}\label{AS-degrees}
For an algebraically stable rational map $f_X:X\dasharrow X$ which is a lift of a birational map $f:\mbP^N\dasharrow\mbP^N$ to a blow-up variety $X$, the degree $\deg(f^n)$ equals  the $(H,H)$ entry of the matrix $(f_X^*)^n$ relative to the basis $H,\mathcal{E}_1,...,\mathcal{E}_n$.
\end{pro}

We now give a computational interpretation of this result. For a homogeneous polynomial $P$ on $\mbC^{N+1}$, denote by $[P]$ the proper transform of the divisor $\{P=0\}\subset\mbP^N$ with respect to the modification $\pi:X\to\mbP^N$. In the above mentioned basis of $\Pic(X)$, we have an expansion for the divisor class of $[P]$:
\begin{equation}\label{div P}
[P]=\deg(P)H-\sum_{i=1}^{n}\mu_i(P)\mathcal{E}_i
\end{equation}
(this can be considered as the definition of the coefficients $\mu_i(P)$). 

\begin{pro}\label{pull-back divisor}
Suppose that an algebraically stable modification $f_X:X\dasharrow X$ of a birational map $f:\mbP^N\dasharrow\mbP^N$ is an automorphism. Then for the proper pull-back $\widetilde{P}$ of a homogeneous polynomial $P$ (see \eqref{def Ptilde}), we have:
\begin{equation}
    [\widetilde{P}]=f_X^*[P].
\end{equation}
\end{pro}
\begin{proof} 
After lifting $f$ to an automorphism $f_X$, formula \eqref{def Ptilde} could be interpreted as a correspondence between divisors under the pull back of $f_X$:
$$
[f^*P]=[\widetilde P]+\sum_{i=1}^n\nu_i[K_i]+\text{linear combination of exceptional divisors}.
$$
By linearity of $f_X^*$, it suffices to consider the case when $P$ is an irreducible polynomial. 
In this case the proper transform $[P]$ is an irreducible divisor in $X$, whose pull back under $f_X^*$ must be an irreducible divisor as well, since $f_X$ is an automorphism. Thus, $f_X^*[P]$ must be an irreducible component of the right-hand side depending on $P$. The only such component is $[\widetilde P]$.
\end{proof}

\begin{cor}\label{cor system d mu}
For the sequence of iterated proper pull-backs of a homogeneous polynomial $P_0$, we have:
$$
[P_{n+1}]=f_X^*[P_n].
$$
As a consequence, the matrix of the linear system which expresses the coefficients $\deg(P_{n+1})$, $\mu_i(P_{n+1})$ through the coefficients $\deg(P_n)$, $\mu_i(P_n)$ is nothing but the matrix of the map $f_X^*:\Pic(X)\to\Pic(X)$ relative to the above mentioned basis, conjugated by ${\rm diag}(1,-1,\ldots,-1)$.
\end{cor}

\begin{re}\label{remark automorphism}
In dimension $N=2$, the Picard group carries an additional structure of a non-degenerate scalar product of the signature $(1,n)$, the so called intersection number, defined by
$$
H\cdot H=1,\quad H\cdot \mathcal E_i=0, \quad \mathcal E_i\cdot \mathcal E_j=-\delta_{ij}.
$$ 
In terms of the intersection product of divisors, the definition of the coefficients in \eqref{div P} can be read as follows:
$$
\deg(P)=[P]\cdot H, \quad \mu_i(P)=[P]\cdot \mathcal{E}_i.
$$
If $f_X$ is an automorphism (this is equivalent to the absence of contracted curves), then $f_X^*$ preserves the intersection product, i.e., is Minkowski-orthogonal, and the matrix of $f_X^*$, conjugated by ${\rm diag}(1,-1,\ldots,-1)$, is nothing but the transposed matrix of $(f_X)_*$ relative to the basis $H,\mathcal{E}_1,...,\mathcal{E}_n$.
\end{re}

\begin{re}
The statement (and the proof) of Proposition \ref{pull-back divisor} remain valid, if $f$ can be lifted to a pseudo-automorphism (recall that a birational map $f_X:X\dasharrow X$ is called a pseudo-automorphism if neither $f_X$ nor $f_X^{-1}$ contract varieties of codimension 1).
\end{re}


\section{Comparison of  two methods for computing $\deg(f^n)$}
\label{sect two methods}

Of course, the two methods for computing $\deg(f^n)$ presented in the previous two sections are intimately related, though not always quite directly.
To clarify the relation, we give an interpretation of the coefficients $\mu_i(P)$ as indices related to elementary blow-ups. It will be  convenient to change the enumeration of the exceptional divisors $\mathcal E_i$ in the following way. 

For a component $\{K=0\}\subset\mathcal C(f)$ contracted to a point $p$, the resolution $\phi$ of this singularity can be obtained as a superposition of elementary blow-ups $\pi{(k)}: X^{(k)}\to X^{(k-1)}$, $k=1,\ldots,m$, with $X^{(0)}=\mbP^N$ (in the case $N=2$ they all are sigma-processes, i.e., point blow-ups):
$$
\phi=\pi^{(1)}\circ\ldots\circ \pi^{(m)}.
$$
 Set
$$
\mathcal E^{(k)}=\text{total transform of the center of the elementary blow-up}\;\; \pi^{(k)}, \quad k=1,\ldots,m. 
$$
The union of $\{\mathcal E^{(k)}\}_{k=1,\ldots,m}$ over all components of the critical set contracted to points, coincides with the set $\{\mathcal E_i\}_{i=1,\ldots,n}$.

Let the blow-ups $\pi^{(k)}$ be given by affine charts
\begin{align*}
[1:x_1:x_2:\ldots:x_N] & =\pi^{(1)}(u_1^{(1)},\ldots,u_N^{(1)}) \;\text{(provided}\;\; x_0\neq 0 \;\;\text{at}\;\; p),\\
(u_1^{(k-1)},\ldots, u_N^{(k-1)}) & =\pi^{(k)}(u_1^{(k)},\ldots,u_N^{(k)}), \quad k=2,\ldots,m,
\end{align*}
with the (local) exceptional divisor defined by $\{u_1^{(k)}=0\}$. Define the indices $\mu_k(P)$ by
\begin{align}
P(\pi^{(1)}(u_1^{(1)},\ldots,u_N^{(1)})) & =(u_1^{(1)})^{\mu_1(P)} P^{(1)}(u_1^{(1)},\ldots,u_N^{(1)}), \label{local blow up index 1} \\
P^{(k-1)}(\pi^{(k)}(u_1^{(k)},\ldots,u_N^{(k)})) & =(u_1^{(k)})^{\mu_k(P)} P^{(k)}(u_1^{(k)},\ldots,u_N^{(k)}), \quad k=2,\ldots,m,  \label{local blow up index k}
\end{align}
where we always assume that $P^{(k)}(0,u_2^{(k)},\ldots,u_N^{(k)})\not\equiv 0$.
\begin{pro}
The indices $\mu_1(P), \ldots,\mu_m(P)$ defined above coincide with the coefficients by $\mathcal E^{(1)},\ldots ,\mathcal E^{(m)}$ in the expansion \eqref{div P}.
\end{pro}
\begin{proof} We interpret the formulas \eqref{local blow up index 1}, \eqref{local blow up index k} as expressing the total transform (with respect to $\pi^{(k)}$) of the divisor $\{P^{(k-1)}=0\}\subset X^{(k-1)}$ through the proper transform described as $\{P^{(k)}=0\}\subset X^{(k)}$ and the exceptional divisor associated to the blow-up $\pi^{(k)}$. For the equivalence classes of these divisors we find:
\begin{small}
\begin{eqnarray*}
\lefteqn{{\big[\text{proper transform of} \;\; \{P=0\}\subset \mbP^N} \;\;\text{w.r.t.}\;\; \pi^{(1)} \big]} \\
& = & \big[\text{total transform of}\;\; \{P=0\}\subset \mbP^N \;\;\text{w.r.t.}\;\; \pi^{(1)}\big] - \mu_1(P)\cE_1, \\ \\
\lefteqn{\big[\text{proper transform of}\;\; \{P^{(k-1)}=0\}\subset X^{(k-1)} \;\;\text{w.r.t.} \;\;\pi^{(k)}\big]} \\
& = & \big[\text{total transform of} \;\;\{P^{(k-1)}=0\}\subset X^{(k-1)} \;\;\text{w.r.t.} \;\;\pi^{(k)}\big] - \mu_k(P)\cE_k\\
& = & \ldots \; \text{(induction)} \\
& = & \big[\text{total transform of}\;\; \{P=0\}\subset \mbP^N\;\; \text{w.r.t.} \;\;\pi^{(1)}\circ\ldots\circ\pi^{(k)}\big] - \mu_k(P)\cE_k-\ldots-\mu_1(P)\cE_1\\
& = & \deg(P)H- \mu_k(P)\cE_k-\ldots-\mu_1(P)\cE_1.
\end{eqnarray*}
\end{small}Here, for example $\cE_1$ in the first equation stands for the exceptional divisor of the blow $\pi^{(1)}$, and in the further equations, by an abuse of notation, for the total transforms of this subvariety with respect to all further $\pi^{(k)}$.
\end{proof}
For a simple singularity (with $m=1$) the index $\mu_1(P)$ coincides with the previously defined $\nu_\phi(P)$ (and with the multiplicity of the point $p$ on the divisor $\{P=0\}$). In more complicated situations, the index $\nu_\phi(P)$ is an integer linear combination of $\mu_1(P),\ldots,\mu_m(P)$ whose coefficients depend on the details of the blow-up structure. See, for instance, Proposition \ref{prop Ex2 nu from mu} for the Example 2 below.

This relation between indices $\mu_i(P)$ and $\nu_\phi(P)$ gives an insight into the relation between the two methods for computing $\deg(f^n)$. If the algebraically stable modification $f_X$ is an automorphism, then the first method gives, as a rule, a smaller system of recurrent equations for $d(n)$ and $\nu_{\phi}(P)$, since it only works with the ``top level'' indices $\nu_{\phi}$, while the second method works with all indices $\mu_i$ for the elementary blow-ups. One can consider the smaller system as a consequence of the bigger one, naturally producing the same sequences $d(n)$. However, if $f_X$ is not an automorphism, the relation is more delicate. Our method requires to blow up all critical components, also non-degree-lowering ones. Under certain circumstances, this could mean even infinitely many blow-ups, compare Example 3 in Section \ref{sect linearizable}. It can be expected that the resulting system of recurrent relations, being finite due to algebraic stability, can be reduced to the system from the second method, but we will not show this in general.


\section{Example 1: an integrable quadratic family}
\label{sect example 1}

The goal of the example of this section is to illustrate the simplest situation of a birational map of $\mbP^2$ which can be modified to a surface automorphism, where the modification only requires to resolve simple singularities, so that all involved blow-ups are just  sigma-processes at distinct (not infinitesimally near) points.

\subsection{Definition of the map}
\label{section Ex1 def}

A fundamental birational planar map is the \emph{quadratic involution} $\sigma: \mbP^2\dasharrow\mbP^2$ given in homogeneous coordinates by
\begin{equation}
\sigma: [y_1:y_2:y_3]\mapsto [1/y_1:1/y_2:1/y_3]=[y_2y_3:y_1y_3:y_1y_2].
\end{equation}
We have:
\begin{equation}
\mathcal I(\sigma)=\{q_1, q_2, q_3\},\quad \mathcal C(\sigma)=\ell_1\cup\ell_2\cup\ell_3,
\end{equation}
where
\begin{equation}
q_1=[1:0:0], \quad q_2=[0:1:0], \quad q_3=[0:0:1],
\end{equation}
and 
\begin{equation}
\ell_1=\{y_1=0\}=(q_2q_3), \quad \ell_2=\{y_2=0\}=(q_1q_3), \quad \ell_3=\{y_3=0\}=(q_1q_2).
\end{equation}

Moreover, $\sigma$ blows up the points $q_i$ to the lines $\ell_i$, and blows down the lines $\ell_i$ to the points $q_i$. If one performs the standard sigma-processes at the points $q_i$ with exceptional divisors $E_i$, then the lift of $\sigma$ to the resulting blow-up surface $X$ is an automorphism exchanging the proper images of $\ell_i$ with $E_i$.

While the involution $\sigma$ is dynamically trivial, the following family contains highly non-trivial dynamical systems:
\begin{equation}\label{quadratic Cremona} 
f=L_1\circ\sigma\circ L_2^{-1},
\end{equation}
where $L_1$, $L_2$ are linear projective automorphisms of $\mbP^2$. According to \cite{Cerveau_Deserti_2013}, formula \eqref{quadratic Cremona} represents a generic planar birational map of degree 2. See also \cite{Bedford_Kim_2004,Diller2011,Blanc_Heden_2015} for general studies of maps of this class and their higher dimensional generalizations.

Clearly, for such a map
\begin{equation}
 \mathcal I(f)=\{L_2q_1, L_2q_2, L_2q_3\}, \quad  \mathcal C(f)=L_2\ell_1\cup L_2\ell_2\cup L_2\ell_3,
\end{equation}  
while
\begin{equation}
\mathcal I(f^{-1})=\{L_1q_1, L_1q_2, L_1q_3\},  \quad  \mathcal C(f^{-1})=L_1\ell_1\cup L_1\ell_2\cup L_1\ell_3.
\end{equation}
Moreover, $f$ blows up the points $L_2q_i$ to the lines $L_1\ell_i$, and blows down the lines $L_2\ell_i$ to the points $L_1q_i$. If we denote $A_i:=L_1q_i$ and $C_i:=L_2q_i$, then we have: $L_1\ell_i=(A_jA_k)$ and $L_2\ell_i=(C_jC_k)$, where $\{i,j,k\}=\{1,2,3\}$.

We consider here the dynamics of maps of the class \eqref{quadratic Cremona}  characterized by the following condition: \emph{all three critical lines $L_2\ell_i=(C_jC_k)$ are degree lowering curves whose images hit the corresponding $C_i=L_2q_i\in \mathcal I(f)$ after three iterations each:}
\begin{equation}\label{2d-Ex1-SingConf}
f: \; (C_jC_k) \,\rightarrow\, A_i \, \rightarrow \, B_i\, \rightarrow \, C_i \,\rightarrow \, (A_jA_k), \quad i=1,2,3, \quad \{i,j,k\}=\{1,2,3\}.
\end{equation}
In the terminology of \cite{Diller2011}, the orbit type of such a map is given by the triple of orbit lengths $(n_1,n_2,n_3)=(3,3,3)$ and the permutation $(\sigma_1,\sigma_2,\sigma_3)=(1,2,3)$. So, the condition is
$$
(L_1\sigma L_2^{-1})(L_1\sigma L_2^{-1})L_1q_i=L_2q_i, \quad i=1,2,3.
$$
This condition is equivalent to
$$
(L_2^{-1}L_1) \sigma (L_2^{-1}L_1) \sigma (L_2^{-1}L_1)={\rm id}.
$$

A geometric construction of a family of maps satisfying this condition, based on Manin involutions for pencils of cubic curves, was proposed in \cite{PSWZ2021}. In this construction, the six points $A_i$ and $C_i$ lie on a conic, while the three points $B_i$ are defined by
$$
B_1=(A_2C_3)\cap (A_3C_2), \quad B_2=(A_3C_1)\cap (A_1C_3), \quad B_3=(A_1C_2)\cap (A_2C_1),
$$
and are collinear by the Pascal theorem. Presumably, all maps \eqref{quadratic Cremona} with the singularity confinement patterns as in \eqref{2d-Ex1-SingConf} are covered by this geometric construction. An early example of such maps was found in \cite{PenroseSmith}:
$$
f:\;\left[\begin{array}{l} x_0 \\ x_1 \\ x_2 \end{array}\right] \; \mapsto \; \left[\begin{array}{l} x_0(x_0+ax_1+a^{-1}x_2) \\ x_1(x_1+ax_2+a^{-1}x_0) \\ x_2(x_2+ax_0+a^{-1}x_1) \end{array}\right].
$$
One can easily check that  this map can be represented in the form \eqref{quadratic Cremona} with
$$
L_1=\begin{pmatrix} 0 & -a & 1 \\ 1 & 0 & -a \\ -a & 1 & 0 \end{pmatrix}, \quad 
L_2=\begin{pmatrix} 0 & -1 & a \\ a & 0 & -1 \\ -1 & a & 0 \end{pmatrix}.
$$

\subsection{Degree computation based on local indices}

Let us blow up $\mbP^2$ at the nine points $A_i$, $B_i$, $C_i$, $i=1,2,3$, via standard sigma-processes and denote these sigma-processes by $\phi_i$, $\phi_{i+3}$ and $\phi_{i+6}$, and the corresponding exceptional divisors by $\cE_i$, $\cE_{i+3}$ and $\cE_{i+6}$, respectively. We obtain a surface $X$, such that the map $f$ is lifted to an automorphism $f_X$ on $X$. Define, according to \eqref{def Pn},
\begin{equation}\label{2d-Ex1_Pn}
    f^*P_n=K_1^{\nu_1(n)}K_2^{\nu_2(n)}K_3^{\nu_3(n)}P_{n+1},
\end{equation}
where $K_i$ is the (linear) defining polynomial of the line $L_2\ell_i=(C_jC_k)$. Here, $\nu_i(n)$ are the indices of the polynomial $P_n$ with respect to the blow-up $\phi_i$, $i=1,2,3$. We adopt the same definition for $\nu_i(n)$ with $i=4,\ldots,9$, and denote also
$$
d(n)=\deg(P_n).
$$
\begin{theo} \label{theorem Ex1 degrees}
The degrees $d(n)$ and indices $\nu_i(n)$ satisfy the following recurrent relations: 
\begin{eqnarray} 
d(n+1) & = & 2d(n)-\nu_1(n)-\nu_2(n)-\nu_3(n),   \label{2d-Ex1-deg} \\
\nu_1(n+1) & = & \nu_4(n),  \label{2d-Ex1-deg-1 1} \\
\nu_2(n+1) & = & \nu_5(n),  \label{2d-Ex1-deg-1 2} \\
\nu_3(n+1) & = & \nu_6(n),  \label{2d-Ex1-deg-1 3} \\
\nu_4(n+1) & = & \nu_7(n),  \label{2d-Ex1-deg-1 4} \\
\nu_5(n+1) & = & \nu_8(n),  \label{2d-Ex1-deg-1 5} \\
\nu_6(n+1) & = & \nu_9(n),  \label{2d-Ex1-deg-1 6} \\
\nu_7(n+1) & = & d(n)-\nu_2(n)-\nu_3(n), \label{2d-Ex1-deg-2 1} \\
\nu_8(n+1) & = & d(n)-\nu_1(n)-\nu_3(n), \label{2d-Ex1-deg-2 2} \\
\nu_9(n+1) & = & d(n)-\nu_1(n)-\nu_2(n). \label{2d-Ex1-deg-2 3} 
\end{eqnarray}
If $P_0$ is a generic linear polynomial, then 
\begin{equation}\label{2d-Ex1 dn}
d(n)=\deg(f^n)=\frac{3}{4}n^2+\frac{9+(-1)^{n+1}}{8}.
\end{equation}
\end{theo}
\begin{proof} Equation \eqref{2d-Ex1-deg} follows directly from \eqref{2d-Ex1_Pn}, and \eqref{2d-Ex1-deg-1 1}--\eqref{2d-Ex1-deg-1 6} are clear, because $f$ acts biregularly at the points $A_i$ and $B_i$. It remains to prove \eqref{2d-Ex1-deg-2 1}--\eqref{2d-Ex1-deg-2 3}. Since all three relations are proved in the same way, we restrict ourselves to the first one of them. We have to determine the behavior of the polynomial
\begin{equation}\label{2d-Ex1 P n+1}
P_{n+1}(x,y,z)=K_1^{-\nu_1(n)}K_2^{-\nu_2(n)}K_3^{-\nu_3(n)}P_n\big(f(x,y,z)\big)
\end{equation}
upon the blow-up $\phi_7$ at $C_1=L_2q_1$, as $u_7\to 0$. 
We can take 
$$
\phi_7(u_7,v_7)=L_2\begin{bmatrix} 1\\ u_7\\ u_7v_7 \end{bmatrix} \quad \Rightarrow \quad 
\tilde f\circ \phi_7=L_1\begin{bmatrix} u_7^2v_7 \\ u_7v_7 \\ u_7\end{bmatrix}
$$
Thus, three components of $\tilde f\circ \phi_7$ are simultaneously  divisible by $u_7$ (and by no higher degree of $u_7$). Therefore, 
$$
P_n(\tilde f\circ \phi_7)\sim u_7^{\deg(P_n)}.
$$ 
Here and below, the symbol $\sim$ relates two quantities differing by a non-vanishing factor (in this case, as $u_7\to 0$). 
The line $(C_2C_3)=\{K_1=0\}$ does not pass through $C_1$, while the lines $(C_1C_3)=\{K_2=0\}$ and $(C_1C_2)=\{K_3=0\}$ do pass through $C_1$, therefore
$$
K_1\circ \phi_7\sim 1, \quad K_2\circ \phi_7\sim u_7, \quad K_3\circ \phi_7\sim u_7.
$$
Collecting all the results, we arrive at 
$$
P_{n+1}\circ \phi_7 \sim u_7^{\deg(P_n)-\nu_2(n)-\nu_3(n)}.
$$
This proves \eqref{2d-Ex1-deg-2 1}. Taking $P_0$ to be a defining polynomial of a generic line, i.e., $d(0)=1$ and $\nu_i(0)=0$ for $i=1,\ldots,9$, we obtain \eqref{2d-Ex1 dn}, which can be shown by induction. 
\end{proof}

\begin{re} In system \eqref{2d-Ex1-deg}--\eqref{2d-Ex1-deg-2 3}, one can eliminate all local indices, and obtain the following recurrence relation for the degrees of $P_n$:
$$
    d(n+1)-2d(n)+2d(n-2)-d(n-3)=0.
$$
\end{re}

\subsection{Degree computation based on Picard group}

The map $f$ becomes algebraically stable, and, moreover, an automorphism, under the lift to the surface $X$. Denoting by $H$ the divisor class of (the proper transform of) a generic line, 
the singularity confinement patterns could be rewritten in terms of divisor classes as:
\begin{equation}\label{2d-Ex1-SingConf-div}
    \begin{split}
H-\cE_8-\cE_9 \; \rightarrow \;  \cE_1 \; \rightarrow  \; \cE_4 \; \rightarrow  \; \cE_7 \; \rightarrow  \; H-\cE_2-\cE_3,\\
H-\cE_7-\cE_9 \; \rightarrow  \; \cE_2  \; \rightarrow  \; \cE_5 \; \rightarrow  \; \cE_8  \; \rightarrow  \; H-\cE_1-\cE_3,\\
H-\cE_7-\cE_8 \; \rightarrow  \; \cE_3 \; \rightarrow  \; \cE_6 \; \rightarrow \;  \cE_9 \; \rightarrow \;  H-\cE_1-\cE_2.\\
    \end{split}
\end{equation}
This follows from \eqref{2d-Ex1-SingConf}, if one takes into account that the divisor class of the proper transform of the line $(C_jC_k)$ to $X$ is 
$H-\cE_{j+6}-\cE_{k+6}$, while the divisor class of the proper transform of the line $(A_jA_k)$ to $X$ is  $H-\cE_{j}-\cE_{k}$. From \eqref{2d-Ex1-SingConf-div}, one can compute the action $(f_X)_*$ of the lifted map $f_X$ on the Picard group $\Pic(X)$:
\begin{equation}\label{Ex1 Picard}
(f_X)_*: \;\left( \begin{array}{cccccccccc}
H  \\
\cE_1 \\  \cE_2 \\  \cE_3\\
\cE_4 \\  \cE_5 \\  \cE_6 \\ 
\cE_7 \\  \cE_8 \\  \cE_9
\end{array} \right)
\rightarrow
\left( \begin{array}{cccccccccc}
2H-\cE_1-\cE_2-\cE_3\\
\cE_4 \\ \cE_5 \\ \cE_6 \\
\cE_7 \\ \cE_8 \\ \cE_9 \\
H-\cE_2-\cE_3 \\ H-\cE_1-\cE_3 \\ H-\cE_1-\cE_2\\
\end{array} \right).
\end{equation}
According to Remark \ref{remark automorphism}, the matrix of the linear system \eqref{2d-Ex1-deg}--\eqref{2d-Ex1-deg-2 3} is transposed to the matrix of $(f_X)_*$ (recall that for simple singularities we have $\nu_i(P)=\mu_i(P)$). In other words, the system \eqref{2d-Ex1-deg}--\eqref{2d-Ex1-deg-2 3} can be directly read off  the formulas \eqref{Ex1 Picard} for $(f_X)_*$.

\subsection{Finding an integral of motion}
We can find the integral of motion of the map $f$ by considering the eigensystem of the matrix $(f_X)_*$ acting on the Picard group $\Pic(X)$. To find (a candidate for) an integral of motion, we look for a divisor class $K$ which is an eigenvector corresponding to the eigenvalue 1, and the linear system $|K|$ has dimension 1. The eigenspace corresponding to the eigenvalue 1 is three-dimensional and is spanned by $H-\cE_1-\cE_4-\cE_7$, $H-\cE_2-\cE_5-\cE_8$ and $H-\cE_3-\cE_6-\cE_9$. It remains to find a linear combination of these divisor classes with a non-trivial linear system. One can see that this is satisfied for the divisor class $3H-\cE_1-\cE_2-\cE_3-\cE_4-\cE_5-\cE_6-\cE_7-\cE_8-\cE_9$ (the anticanonical divisor class $-K_X$ of the surface $X$). Its linear system corresponds to the pencil of cubic curves in $\mbP^2$ with nine base points $A_i$, $B_i$, $C_i$. The quotient of any two cubic polynomials of the pencil is easily checked to be an integral of $f$.


\section{Example 2: Invariant plane at infinity for dPI}
\label{sect example 2}

The goal of the example of this section is to illustrate a more complicated situation of a birational map of $\mbP^2$ which can be modified to a surface automorphism, where the modification requires to resolve a number of infinitely near singularities.

\subsection{Definition of the map}

The example comes from \cite{Ercolani2021}, in which the discrete Painlevé I (dPI) equation is transformed into  a three-dimensional autonomous discrete map that has an invariant plane. The restriction of this map to the invariant plane is the following planar birational map:
$$
F:(x,y)\mapsto\Big(\frac{y}{y-1},\frac{x}{(y-1)^2}\Big).
$$

We rewrite this map in homogeneous coordinates:
\begin{equation}\label{dP1 2D f}
f[x:y:z]= [y(y-z):xz:(y-z)^2].
\end{equation}
The inverse map is given by
$$
f^{-1}[x:y:z]=[yz:x(x-z):(x-z)^2].
$$
The singularities of the forward map $f$ and the inverse map are
$$
{\mathcal I}(f)=\{p_3=[0:1:1],\; p_4=[1:0:0]\}
$$ 
and 
$$
{\mathcal I}(f^{-1})=\{p_1=[0:1:0], \; p_2=[1:0:1]\}.
$$
The respective critical sets are:
$$
{\mathcal C}(f)=\{z=0\}\cup\{y-z=0\}
$$ 
and 
$$
{\mathcal C}(f^{-1})=\{z=0\}\cup\{x-z=0\}.
$$
Both components of ${\mathcal C}(f)$ are degree lowering curves:
\begin{equation}\label{ex2-singconf-1}
    f: \quad \{z=0\} \;\rightarrow \; p_2 \;\rightarrow \; p_3  \; \in \; {\mathcal I}(f).
\end{equation}
and
\begin{equation}\label{ex2-singconf-2}
    f: \quad \{y-z=0\} \;\rightarrow \; p_1 \;\rightarrow \; p_2 \;\rightarrow \; p_3 \; \in \; {\mathcal I}(f).
\end{equation}

\subsection{Lifting $f$ to a surface automorphism}

Similarly as in the previous example, we can resolve the singularities of $f$ by successive blowing ups, and lift $f$ to a surface automorphism. However, unlike the previous example, where all the singularities are simple, some singularities of $f$ are deeper (sometimes called \emph{``infinitely near singularities''} in the literature), and more complicated blow-ups are needed to resolve them. So we will work out the details for the blow-ups.  

Our goal is to lift $f$ to an automorphism on some surface $X$ obtained from $\mbP^2$ by a sequence of blow-ups. To achieve this, we need to consider the two orbits \eqref{ex2-singconf-1} and \eqref{ex2-singconf-2}, and to blow up a point whenever a curve is contracted to this point. It turns out we can achieve this by ten blow-ups: 

{\bf Action of $f$ on the critical line $\{y-z=0\}$.} This line is contracted by $f$ to the point $p_1=[0:1:0]$. In order to resolve this, we  blow up this point. We use the following change to coordinates $(u_1,v_1)$ in the affine chart $y=1$:
\begin{equation}\label{dP1 2D E1}
\pi_1:\;\left\{\begin{array}{l} x=u_1, \\ y=1, \\ z=u_1v_1.\end{array}\right.
\end{equation}

The exceptional divisor of $\pi_1$ is given by $E_1=\{u_1=0\}$. In coordinates \eqref{dP1 2D E1} for the image $[\widetilde x:\widetilde y:\widetilde z]$, we compute:
$$
\widetilde u_1=\frac{\widetilde{x}}{\widetilde{y}}=\frac{y(y-z)}{xz}, \quad \widetilde v_1=\frac{\widetilde z}{\widetilde x}=\frac{y-z}{y}.
$$
It follows that the image of (the proper transform of) the line $\{y-z=0\}$ in the exceptional divisor $E_1$ is the point $(0,\widetilde v_1)=(0,0)$, and a further blow-up is necessary. For this second blow-up we use the following change to coordinates $(u_2,v_2)$:
\begin{equation}\label{dP1 2D E2}
\pi_2:\;\left\{\begin{array}{l} u_1=u_2, \\ v_1=u_2v_2\end{array}\right. \quad \Rightarrow\quad 
\phi_2=\pi_1\circ\pi_2:\;\left\{\begin{array}{l} x=u_2, \\ y=1, \\ z=u_2^2v_2. \end{array}\right.
\end{equation}
The exceptional divisor of $\pi_2$ is given by $E_2=\{u_2=0\}$. In coordinates \eqref{dP1 2D E2} for the image $[\widetilde x:\widetilde y:\widetilde z]$, we compute:
$$
\widetilde u_2=\frac{\widetilde{x}}{\widetilde{y}}=\frac{y(y-z)}{xz}, \quad \widetilde v_2=\frac{\widetilde z\widetilde y}{\widetilde x^2}=\frac{xz}{y^2}.
$$
As soon as $y-z=0$, we have $\widetilde u_2=0$ and $\widetilde v_2=x/y$. Thus, (the proper transform of) the line $\{y-z=0\}$ is biregularly mapped to the exceptional divisor $E_2$:
$$
f|_{\{y-z=0\}}:[x:y:y] \mapsto (0,\widetilde v_2)=(0,x/y)\in E_2.
$$

{\bf Action of $f$ on the critical line $\{z=0\}$.}  This line is contracted by $f$ to the point $p_2=[1:0:1]$. In order to resolve this, we  blow up this point. We use the following change to coordinates $(u_3,v_3)$ in the affine chart $x=1$:
\begin{equation}\label{dP1 2D E3}
\phi_3=\pi_3:\;\left\{\begin{array}{l} x=1, \\ y=u_3v_3, \\ z=1-u_3.\end{array}\right.
\end{equation}
The exceptional divisor of $\pi_3$ is given by $E_3=\{u_3=0\}$. In coordinates \eqref{dP1 2D E3} for the image $[\widetilde x:\widetilde y:\widetilde z]$, we compute:
$$
\widetilde u_3=1-\frac{\widetilde{z}}{\widetilde{x}}=1-\frac{y-z}{y}=\frac{z}{y}, \quad \widetilde v_3=\frac{\widetilde y}{\widetilde{x}-\widetilde{z}}=\frac{xz}{z(y-z)}=\frac{x}{y-z}.
$$ 
As soon as $z=0$, we have $\widetilde u_3=0$ and $\widetilde v_3=x/y$. Thus, (the proper transform of) the line $\{z=0\}$ is biregularly mapped to the exceptional divisor $E_3$:
$$
f|_{\{z=0\}}:[x:y:0]\mapsto (0,\widetilde v_3)=(0,x/y)\in E_3.
$$
\smallskip

{\bf Action of $f$ on $E_2$.} Recall that we parametrize a neighborhood of $E_2$ by \eqref{dP1 2D E2}. We compute:
\begin{eqnarray*}
f\left[\begin{array}{c} u_2 \\ 1 \\ u_2^2v_2\end{array}\right] & = & \left[\begin{array}{c} 1-u_2^2v_2\\ u_2^3v_2 \\ (1-u_2^2v_2)^2 \end{array}\right]=\left[\begin{array}{c}  1 \\ \dfrac{u_2^3v_2}{1-u_2^2v_2} \\  1-u_2^2v_2  \end{array}\right]=:\left[\begin{array}{c}  \widetilde{x} \\  \widetilde{y} \\ \widetilde{z} \end{array}\right].
\end{eqnarray*}
Thus, $f$ contracts $E_2=\{u_2=0\}$ to the point $p_2=[1:0:1]$. In order to resolve this, we blow-up this point as in \eqref{dP1 2D E3}. We compute:
$$
\widetilde{u}_3\widetilde{v}_3=\dfrac{u_2^3v_2}{1-u_2^2v_2}, \quad 1-\widetilde{u}_3=1-u_2^2v_2,
$$
so that 
$$
\widetilde{u}_3=u_2^2v_2, \quad \widetilde{v}_3=\dfrac{u_2}{1-u_2^2v_2}.
$$
Thus, $f$ contracts the exceptional divisor $E_2=\{u_2=0\}$ to a single point $(u_3,v_3)=(0,0)\in E_3$. In order to resolve this, we  blow up this point using the following change to coordinates $(u_4,v_4)$:
\begin{equation}\label{dP1 2D E4}
\pi_4:\;\left\{\begin{array}{l} u_3=u_4v_4, \\ v_3=u_4, \end{array}\right. \quad \Rightarrow \quad 
\pi_3\circ\pi_4:\;\left\{\begin{array}{l} x=1, \\ y=u_4^2v_4, \\ z=1-u_4v_4. \end{array}\right.
\end{equation}
The exceptional divisor of $\pi_4$ is given by $E_4=\{u_4=0\}$. In coordinates \eqref{dP1 2D E4} for the expressions on the right-hand side, we compute:
$$
\widetilde{u}_4^2\widetilde{v}_4=\dfrac{u_2^3v_2}{1-u_2^2v_2}, \quad 1-\widetilde{u}_4\widetilde{v}_4=1-u_2^2v_2,
$$
so that
$$
\widetilde{u}_4=\frac{u_2}{1-u_2^2v_2}, \quad \widetilde{v}_4=u_2v_2(1-u_2^2v_2).
$$
Thus, $f$ contracts the exceptional divisor $E_2=\{u_2=0\}$ to a single point $(u_4,v_4)=(0,0)\in E_4$. In order to resolve this, a further  blow-up is necessary, this time using the following change to coordinates $(u_5,v_5)$:
\begin{equation}\label{dP1 2D E5}
\pi_5:\;\left\{\begin{array}{l} u_4=u_5, \\ v_4=u_5v_5,\end{array}\right. \quad \Rightarrow \quad 
\phi_5=\pi_3\circ\pi_4\circ\pi_5:\;
\left\{\begin{array}{l} x=1,\\ y=u_5^3v_5, \\ z=1-u_5^2v_5.\end{array}\right.
\end{equation}
The exceptional divisor of $\pi_5$ is given by $E_5=\{u_5=0\}$. In coordinates \eqref{dP1 2D E5}, we compute: 
$$
\widetilde{u}_5^3\widetilde{v}_5=\dfrac{u_2^3v_2}{1-u_2^2v_2}, \quad 1-\widetilde{u}_5^2\widetilde{v}_5=1-u_2^2v_2,
$$
so that
$$
\widetilde{u}_5=\frac{u_2}{1-u_2^2v_2}, \quad \widetilde{v}_5=v_2(1-u_2^2v_2)^2.
$$
This gives a regular map of a neighborhood of the exceptional divisor $E_2=\{u_2=0\}$ to a neighborhood of the exceptional divisor $E_5=\{u_5=0\}$, 
and in the limit $u_2\to 0$ a biregular map of $E_2$ to $E_5$ given by  
$$
f|_{E_2}: (0,v_2)\mapsto (0,\widetilde{v}_5)=(0,v_2)\in E_5.
$$

{\bf Action of $f$ on $E_3$.}  Recall that we parametrize a neighborhood of $E_3$ according to \eqref{dP1 2D E3}. 
In these coordinates,
\begin{eqnarray*}
f\left[\begin{array}{c} 1 \\ u_3v_3 \\ 1-u_3 \end{array}\right] & = & \left[\begin{array}{c} -u_3v_3(1-u_3-u_3v_3)\\ 1-u_3 \\ (1-u_3-u_3v_3)^2 \end{array}\right]=\left[\begin{array}{c}  -\dfrac{u_3v_3}{1-u_3-u_3v_3} \\  \dfrac{1-u_3}{(1-u_3-u_3v_3)^2} \\1 \end{array}\right]=:\left[\begin{array}{c}  \widetilde{x} \\  \widetilde{y} \\ \widetilde{z} \end{array}\right].
\end{eqnarray*}
We see that $f$ contracts $E_3$ to the point $p_3=[0:1:1]$. In order to resolve this, we  blow up this point. We use the following change to coordinates $(u_6,v_6)$ in the affine chart $z=1$:
\begin{equation}\label{dP1 2D E6}
\phi_6=\pi_6:\;\left\{\begin{array}{l} x=u_6v_6, \\ y=1+u_6, \\ z=1. \end{array}\right.
\end{equation}
The exceptional divisor of $\pi_6$ is given by $E_6=\{u_6=0\}$. We have:
$$
1+\widetilde{u}_6=\dfrac{1-u_3}{(1-u_3-u_3v_3)^2}, \quad \widetilde{u}_6\widetilde{v}_6= -\dfrac{u_3v_3}{1-u_3-u_3v_3}.
$$
This gives a regular map of a neighborhood of $E_3=\{u_3=0\}$ to a neighborhood of $E_6=\{u_6=0\}$, and, in the limit $u_3\to 0$, a regular map of $E_3$ to $E_6$, given by
$$
f|_{E_3}:(0,v_3)\mapsto (0,\widetilde{v}_6)=\big(0,-v_3/(1+2v_3)\big)\in E_6. 
$$

{\bf Action of $f$ on $E_5$.} Recall that we parametrize a neighborhood of $E_5$ by \eqref{dP1 2D E5}. We compute:
\begin{eqnarray*}
f\left[\begin{array}{c} 1 \\ u_5^3v_5 \\ 1-u_5^2v_5 \end{array}\right] & = & \left[\begin{array}{c} u_5^3v_5(1-u_5^2v_5-u_5^3v_5)\\ 1-u_5^2v_5\\ (1-u_5^2v_5-u_5^3v_5)^2 \end{array}\right]=
\left[\begin{array}{c} \dfrac{u_5^3v_5}{1-u_5^2v_5-u_5^3v_5}\\ \dfrac{1-u_5^2v_5}{(1-u_5^2v_5-u_5^3v_5)^2} \\ 1 \end{array}\right].
\end{eqnarray*}
Thus, $f$ contracts $E_5=\{u_5=0\}$ to a single point $p_3=[0:1:1]$. To resolve this, we blow up this point as in \eqref{dP1 2D E6}. Then we find:
$$
\widetilde{u}_6\widetilde{v}_6=\dfrac{u_5^3v_5}{1-u_5^2v_5-u_5^3v_5}, \quad 1+\widetilde{u}_6=\dfrac{1-u_5^2v_5}{(1-u_5^2v_5-u_5^3v_5)^2},
$$
and, upon a simple computation:
$$
\widetilde{u}_6=u_5^2v_5+O(u_5^3), \quad \widetilde{v}_6=u_5+O(u_5^2).
$$
Therefore, $f$ contracts $E_5=\{u_5=0\}$ to a single point $(u_6,v_6)=(0,0)\in E_6$. To resolve this, we blow up this point via the following change to coordinates $(u_7,v_7)$:
\begin{equation}\label{dP1 2D E7}
\pi_7:\;\left\{\begin{array}{l} u_6=u_7v_7, \\ v_6=u_7, \end{array}\right. \quad \Rightarrow \quad 
\pi_6\circ\pi_7:\;\left\{\begin{array}{l} x=u_7^2v_7, \\ y=1+u_7v_7, \\z=1. \end{array}\right.
\end{equation}
The exceptional divisor of $\pi_7$ is given by $E_7=\{u_7=0\}$. We find:
$$
\widetilde{u}_7^2\widetilde{v}_7=\dfrac{u_5^3v_5}{1-u_5^2v_5-u_5^3v_5}, \quad 1+\widetilde{u}_7\widetilde{v}_7=\dfrac{1-u_5^2v_5}{(1-u_5^2v_5-u_5^3v_5)^2}.
$$
A simple computation gives:
$$
\widetilde{u}_7=u_5+O(u_5^2), \quad \widetilde{v}_7=u_5v_5+O(u_5^2).
$$
Therefore, $f$ contracts $E_5=\{u_5=0\}$ to a single point $(u_7,v_7)=(0,0)\in E_7$, and we have to iterate the blow-up procedure, via
\begin{equation}\label{dP1 2D E8}
\pi_8:\;\left\{\begin{array}{l} u_7=u_8, \\ v_7=u_8v_8, \end{array}\right. \quad \Rightarrow \quad 
\phi_8=\pi_6\circ\pi_7\circ\pi_8:\;\left\{\begin{array}{l} x=u_8^3v_8, \\ y=1+u_8^2v_8, \\z=1. \end{array}\right.
\end{equation}
The exceptional divisor of $\pi_8$ is given by $E_8=\{u_8=0\}$. We find:
$$
\widetilde{u}_8^3\widetilde{v}_8=\dfrac{u_5^3v_5}{1-u_5^2v_5-u_5^3v_5}, \quad 1+\widetilde{u}_8^2\widetilde{v}_8=\dfrac{1-u_5^2v_5}{(1-u_5^2v_5-u_5^3v_5)^2}.
$$
And, upon a simple computation:
$$
\widetilde{u}_8=u_5+O(u_5^2), \quad \widetilde{v}_8=v_5+O(u_5).
$$
Thus, $f|_{E_5}$ is biregular:
$$
f|_{E_5}: (0,v_5)\mapsto (0,\widetilde{v}_8)=(0,v_5)\in E_8.
$$

{\bf Action of $f$ on $E_6$.} Recall that we parametrize a neighborhood of $E_8$ by \eqref{dP1 2D E6}. We compute:
\begin{eqnarray*}
f\left[\begin{array}{c} u_6v_6 \\ 1+u_6 \\ 1\end{array}\right] & = & \left[\begin{array}{c} u_6(1+u_6) \\ u_6v_6 \\ u_6^2 \end{array}\right]=\left[\begin{array}{c}  1+u_6 \\ v_6 \\ u_6 \end{array}\right].
\end{eqnarray*}
This gives a regular map of a neighborhood of $E_6=\{u_6=0\}$ to a neighborhood of the line $\{z=0\}$, and, in the limit $u_6\to 0$, a regular map of $E_6$ to this line, given by
$$
f|_{E_6}:(0,v_6)\mapsto[1: v_6:0]\in \{z=0\}.
$$

{\bf Action of $f$ on $E_8$.} Recall that we parametrize a neighborhood of $E_8$ by \eqref{dP1 2D E8}. We compute:
\begin{eqnarray*}
f\left[\begin{array}{c} u_8^3v_8 \\ 1+u_8^2v_8 \\ 1 \end{array}\right] & = & \left[\begin{array}{c} u_8^2v_8(1+u_8^2v_8) \\ u_8^3v_8 \\ u_8^4v_8^2 \end{array}\right]=
\left[\begin{array}{c} 1 \\ u_8/(1+u_8^2v_8) \\ u_8^2v_8/(1+u_8^2v_8) \end{array}\right].
\end{eqnarray*}
Thus, $f$ contracts $E_8=\{u_8=0\}$ to a single point $p_4=[1:0:0]$. To regularize this, we blow-up this point. We use the following change to coordinates $(u_9,v_9)$ in the affine chart $x=1$:
\begin{equation}\label{dP1 2D E9}
\pi_9:\;\left\{\begin{array}{l} x=1, \\ y=u_9, \\ z=u_9v_9. \end{array}\right.
\end{equation}
The exceptional divisor of $\pi_9$ is given by $E_9=\{u_9=0\}$. Then we find:
$$
\widetilde{u}_9=\frac{u_8}{1+u_8^2v_8}, \quad \widetilde{v}_9=\frac{u_8v_8}{1+u_8^2v_8}.
$$
Thus, $f$ contracts $E_8=\{u_8=0\}$ to a single point $(u_9,v_9)=(0,0)\in E_9$, so that we have to iterate the blow-up procedure, via
\begin{equation}\label{dP1 2D E10}
\pi_{10}:\;\left\{\begin{array}{l} u_9=u_{10}, \\ v_9=u_{10}v_{10} \end{array}\right. \quad \Rightarrow\quad 
\phi_{10}=\pi_9\circ\pi_{10}:\;\left\{\begin{array}{l} x=1,\\ y=u_{10}, \\z=u_{10}^2v_{10}.\end{array}\right.
\end{equation}
The exceptional divisor of $\pi_{10}$ is given by $E_{10}=\{u_{10}=0\}$. In coordinates \eqref{dP1 2D E10}, $f$ acts as follows:
$$
\widetilde{u}_{10}=\frac{u_8}{1+u_8^2v_8}, \quad \widetilde{v}_{10}=v_8(1+u_8^2v_8).
$$
Thus, $f$ maps $E_8=\{u_8=0\}$ biregularly to $E_{10}=\{u_{10}=0\}$, 
$$
f|_{E_8}:(0,v_8)\mapsto (0,\widetilde{v}_{10})=(0,v_8).
$$

{\bf Action of $f$ on $E_{10}$.} We parametrize a neighborhood of $E_{10}$ by \eqref{dP1 2D E10}, and compute:
$$
f\left[\begin{array}{c} 1 \\ u_{10} \\ u_{10}^2v_{10}\end{array} \right]  =  
\left[\begin{array}{c} u_{10}^2(1-u_{10}v_{10})\\ u_{10}^2v_{10} \\ u_{10}^2(1-u_{10}v_{10})^2 \end{array}\right]=
\left[\begin{array}{c}  1-u_{10}v_{10} \\ v_{10}\\ (1-u_{10}v_{10})^2 \end{array}\right].
$$
Thus, $f$ maps the exceptional divisor $E_{10}=\{u_{10}=0\}$ biregularly to the line $\{x-z=0\}$, with
$$
f|_{E_{10}}:(0,v_{10})\mapsto[1:v_{10}:1]\in \{x-z=0\}.
$$

The remaining four simple arguments are not relevant for the algorithm for the computation of $\deg(f^n)$, as they treat the action of $f$ on exceptional divisors that are ``lover than top level'', but they are essential for the determination of the induced map on $\Pic(X)$. 

{\bf Action of $f$ on $E_1$.} Recall that we parametrize a neighborhood of $E_1$ by \eqref{dP1 2D E1}. We compute:
\begin{eqnarray*}
f\left[\begin{array}{c} u_1\\ 1 \\ u_1v_1 \end{array}\right] & = & \left[\begin{array}{c} 1-u_1v_1 \\ u_1^2v_1 \\ (1-u_1v_1)^2 \end{array}\right]=\left[\begin{array}{c}  1 \\ u_1^2v_1/(1-u_1v_1) \\  1-u_1v_1  \end{array}\right]=:\left[\begin{array}{c}  \widetilde{x} \\  \widetilde{y} \\ \widetilde{z} \end{array}\right].
\end{eqnarray*}
We identify:
$$
\widetilde{u}_4^2\widetilde{v}_4=\dfrac{u_1^2v_1}{1-u_1v_1}, \quad 1-\widetilde{u}_4\widetilde{v}_4=1-u_1v_1,
$$
so that 
$$
\widetilde{u}_4=\dfrac{u_1}{1-u_1v_1}, \quad \widetilde{v}_4=v_1(1-u_1v_1).
$$
This gives a regular map of a neighborhood of the exceptional divisor $E_1=\{u_1=0\}$ to a neighborhood of the exceptional divisor $E_4=\{u_4=0\}$, 
and in the limit $u_1\to 0$ a regular map of $E_1$ to $E_4$ given by  
$$
f|_{E_1}: (0,v_1)\mapsto (0,\widetilde{v}_4)=(0,v_1)\in E_4.
$$

{\bf Action of $f$ on $E_4$.} Recall that we parametrize a neighborhood of $E_4$ by \eqref{dP1 2D E4}. We compute:
\begin{eqnarray*}
f\left[\begin{array}{c} 1\\ u_4^2v_4 \\ 1- u_4v_4\end{array}\right] & = & \left[\begin{array}{c} u_4^2v_4(1-u_4v_4-u_4^2v_4) \\ 1-u_4v_4\\ (1-u_4v_4-u_4^2v_4)^2 \end{array}\right]= \left[\begin{array}{c} \dfrac{u_4^2v_4}{1-u_4v_4-u_4^2v_4} \\ \dfrac{1-u_4v_4}{(1-u_4v_4-u_4^2v_4)^2}\\ 1 \end{array}\right].
\end{eqnarray*}
We find:
$$
\widetilde{u}_7^2\widetilde{v}_7=\dfrac{u_4^2v_4}{1-u_4v_4-u_4^2v_4}, \quad 1+\widetilde{u}_7\widetilde v_7=\dfrac{1-u_4v_4}{(1-u_4v_4-u_4^2v_4)^2}.
$$
Thus,
$$
\widetilde{u}_7=u_4+O(u_4^2), \quad \widetilde{v}_7=v_4+O(u_4).
$$
Thus, $f$ maps $E_4=\{u_4=0\}$ regularly to $E_7=\{u_7=0\}$, 
$$
f|_{E_4}:(0,v_4)\mapsto (0,\widetilde{v}_7)=(0,v_4)\in E_7.
$$

{\bf Action of $f$ on $E_7$.} Recall that we parametrize a neighborhood of $E_7$ by \eqref{dP1 2D E7}. We compute:
\begin{eqnarray*}
f\left[\begin{array}{c} u_7^2v_7 \\ 1+u_7v_7 \\ 1 \end{array}\right] & = & \left[\begin{array}{c} u_7v_7(1+u_7v_7) \\ u_7^2v_7 \\ u_7^2v_7^2 \end{array}\right]=
\left[\begin{array}{c} 1 \\ u_7/(1+u_7v_7) \\ u_7v_7/(1+u_7v_7) \end{array}\right].
\end{eqnarray*}
We compute:
$$
\widetilde{u}_9=\frac{u_7}{1+u_7v_7}, \quad \widetilde{v}_9=v_7.
$$
Thus, $f$ maps $E_7=\{u_7=0\}$ regularly to $E_9=\{u_9=0\}$, 
$$
f|_{E_7}:(0,v_7)\mapsto (0,\widetilde{v}_9)=(0,v_7)\in E_9.
$$

{\bf Action of $f$ on $E_9$.} We parametrize a neighborhood of $E_9$ by \eqref{dP1 2D E9}, and compute:
\begin{eqnarray*}
f\left[\begin{array}{c} 1 \\ u_9 \\ u_9v_9 \end{array}\right] & = & \left[\begin{array}{c} u_9^2(1-v_9) \\ u_9v_9 \\ u_9^2(1-v_9)^2 \end{array}\right]=
\left[\begin{array}{c} u_9(1-v_9)/v_9 \\ 1 \\ u_9(1-v_9)^2/v_9 \end{array}\right].
\end{eqnarray*}
This is a regular map, if on the right-hand side we use the coordinates \eqref{dP1 2D E1}. In these coordinates, $f$ acts as follows:
$$
\widetilde{u}_1=u_9(1-v_9)/v_9, \quad \widetilde{v}_1=1-v_9.
$$
Thus, $f$ maps $E_9\setminus\{v_9=0\}$ regularly to $E_1\setminus\{v_1=1\}$, 
$$
f|_{E_9}:(0,v_9)\mapsto (0,\widetilde{v}_1)=(0,1-v_9).
$$

Exceptional divisors on the blow-up surface $X$ are shown in Fig. \ref{fig:Divisors}. 

\begin{figure}[h]
\centering
\includegraphics[width=\textwidth]{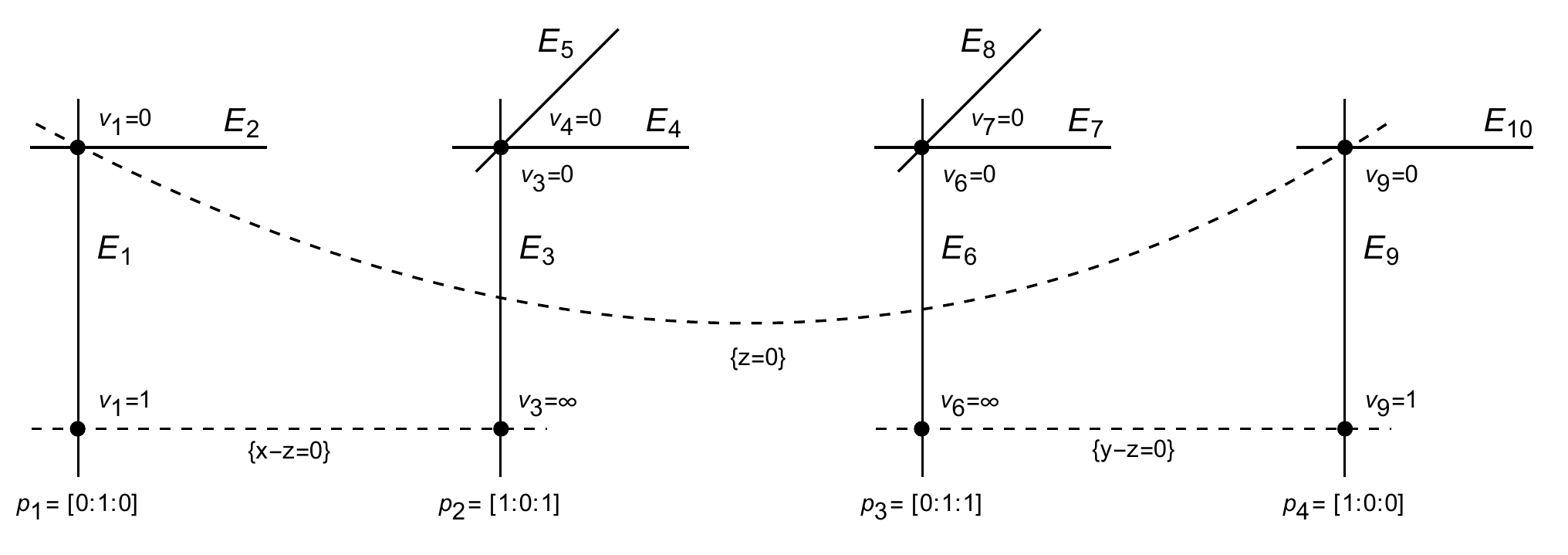}
\caption[Blowing ups]{Divisors appearing in the blowing ups at the four points $p_1,p_2,p_3,p_4$ (from left to right). The total transforms of the blown-up points are: $\mathcal E_1=E_1+E_2$, $\mathcal E_2=E_2$, $\mathcal E_3=E_3+E_4+E_5$, $\mathcal E_4=E_4+E_5$, $\mathcal E_5=E_5$, $\mathcal E_6=E_6+E_7+E_8$, $\mathcal E_7=E_7+E_8$, $\mathcal E_8=E_8$, $\mathcal E_9=E_9+E_{10}$, $\mathcal E_{10}=E_{10}$.}
\label{fig:Divisors}
\end{figure}

\subsection{Degree computation with local indices}

According to our general algorithm, we start with defining local indices with respect to the blow-ups $\phi_2$ and $\phi_3$, given in \eqref{dP1 2D E2} and \eqref{dP1 2D E3}, which regularize the map $f$ contracting the critical lines $\{y-z=0\}$ and $\{z=0\}$ to the points $p_1$ and $p_2$, respectively. The concrete realization of Definition \ref{def index} reads as follows.

\begin{de} \label{definition indices nu}
Let $P(x,y,z)$ be a homogeneous polynomial. The indices of $P$ with respect to the blow-up maps $\phi_2$ and $\phi_3$, denoted by $\nu_2(P)$ and $\nu_3(P)$, are defined by the relations
\begin{equation}\label{def nu2}
P\circ\phi_2=P(u_2,1,u_2^2v_2)=u_2^{\nu_2(P)}\widetilde P(u_2,v_2), \quad \widetilde P(0,v_2)\not\equiv 0,
\end{equation}
respectively 
\begin{equation}\label{def nu3}
P\circ\phi_3=P(1,u_3v_3,1-u_3)=u_3^{\nu_3(P)}\widetilde P(u_3,v_3), \quad \widetilde P(0,v_3)\not\equiv 0.
\end{equation}
The blow-up maps $\phi_2$ and $\phi_3$ are given in \eqref{dP1 2D E2}, resp. \eqref{dP1 2D E3}.
\end{de}
Similarly to $\nu_2$, $\nu_3$, we define the indices $\nu_5$, $\nu_6$, $\nu_8$, $\nu_{10}$, based on formulas \eqref{dP1 2D E5}, \eqref{dP1 2D E6}, \eqref{dP1 2D E8}, \eqref{dP1 2D E10} for the blow-ups maps $\phi_5$, $\phi_6$, $\phi_8$, $\phi_{10}$.  

A concrete realization of Theorem \ref{theorem split off} in the present case reads as follows.

\begin{pro}\label{lemma prefactor z y-z}
For any homogeneous polynomial $P$, the pull-back $f^*P=P\circ\tilde f$ is divisible by $(y-z)^{\nu_2(P)}$ (and by no higher degree of $y-z$) and by $z^{\nu_3(P)}$ (and by no higher degree of $z$).
\end{pro}
\begin{proof}
\begin{enumerate}[1)]
\item In the chart \eqref{dP1 2D E2}, we compute:
$$
P(x,1,z)=u_2^{\nu_2}\widetilde P(u_2,v_2)=x^{\nu_2}\widetilde P\Big(x,\frac{z}{x^2}\Big).
$$
In the homogeneous coordinates,
$$
P(x,y,z)=y^d P\Big(\frac{x}{y},1,\frac{z}{y}\Big)=x^{\nu_2}y^{d-\nu_2}\widetilde P\Big(\frac{x}{y},\frac{yz}{x^2}\Big).
$$
Setting $\tilde f(x,y,z)=(X,Y,Z)$, we compute, according to \eqref{dP1 2D f}:
$$
f^*P(x,y,z)=X^{\nu_2}Y^{d-\nu_2}\widetilde P\Big(\frac{X}{Y},\frac{YZ}{X^2}\Big)=\big(y(y-z)\big)^{\nu_2}(xz)^{d-\nu_2}\widetilde P\Big(\frac{y(y-z)}{xz},\frac{xz}{y^2}\Big).
$$
This polynomial is manifestly divisible by $(y-z)^{\nu_2}$ (and by no higher degree of $y-z$). 
\item Likewise, we compute in the chart \eqref{dP1 2D E3}:
$$
P(1,y,z)=u_3^{\nu_3}\widetilde P(u_3,v_3)=(1-z)^{\nu_3}\widetilde P\Big(1-z,\frac{y}{1-z}\Big).
$$
In the homogeneous coordinates,
$$
P(x,y,z)=x^d P\Big(1,\frac{y}{x},\frac{z}{x}\Big)=(x-z)^{\nu_3}x^{d-\nu_3}\widetilde P\Big(\frac{x-z}{x},\frac{y}{x-z}\Big).
$$
Therefore, upon a straightforward computation according to \eqref{dP1 2D f}:
$$
\begin{aligned}
f^*P(x,y,z)&=(X-Z)^{\nu_3}X^{d-\nu_3}\widetilde P\Big(\frac{X-Z}{X},\frac{Y}{X-Z}\Big)\\&=\big(z(y-z)\big)^{\nu_3}\big(y(y-z)\big)^{d-\nu_3}\widetilde P\Big(\frac{z}{y},\frac{x}{y-z}\Big).
\end{aligned}
$$
This polynomial is manifestly divisible by $z^{\nu_3}$ (and by no higher degree of $z$). 
\end{enumerate}
\end{proof}

Next, we will study the behavior of a given polynomial $P=P_0$ under iterates of the pull-back $f^*$. According to Proposition \ref{lemma prefactor z y-z}, we define recursively:
\begin{equation}\label{dP1 2D iterate P}
P_{n+1}=(y-z)^{-\nu_2(n)}z^{-\nu_3(n)}f^*P_n,
\end{equation}
where $\nu_2(n)$, $\nu_3(n)$ stand for $\nu_2(P_n)$, $\nu_3(P_n)$ (and similarly for other indices). More explicitly,
\begin{equation}\label{dP1 2D L n+1}
P_{n+1}(x,y,z)=(y-z)^{-\nu_2(n)}z^{-\nu_3(n)}P_n\big(y(y-z),xz,(y-z)^2\big).
\end{equation}
As usual, we also set $d(n)=\deg(P_n)$. 
\begin{theo}\label{theorem Ex2 recurrent nu} 
The degrees $d(n)$ and the indices $\nu_i(n)$ satisfy the following recurrent relations:
\begin{eqnarray}
d(n+1) & = & 2d(n)-\nu_2(n)-\nu_3(n), \label{dP1 2D d}\\
\nu_2(n+1) & = & \nu_5(n)-2\nu_3(n),  \label{dP1 2D nu2}\\
\nu_3(n+1) & = & \nu_6(n),  \label{dP1 2D nu3}\\
\nu_5(n+1) & = & \nu_8(n),  \label{dP1 2D nu5} \\    
\nu_6(n+1) & = & d(n)-\nu_2(n), \label{dP1 2D nu6} \\
\nu_8(n+1) & = & 2d(n)-2\nu_2(n)+\nu_{10}(n),  \label{dP1 2D nu8}   \\           
\nu_{10}(n+1) & = &  2d(n)-\nu_2(n)-2\nu_3(n).   \label{dP1 2D nu10}
\end{eqnarray} 
\end{theo}
\begin{proof} Formula \eqref{dP1 2D d} follows directly from \eqref{dP1 2D iterate P}. For all other formulas, we have to study the behavior of the polynomial \eqref{dP1 2D L n+1}. 

$\bullet$ Eqs. \eqref{dP1 2D nu3}, \eqref{dP1 2D nu5} hold true, since $f_X$ acts biregularly in the neighborhood of the exceptional divisors $E_3$ and $E_5$ over the point $p_2=[1:0:1]$, and the factors $z$ and $y-z$ do not vanish at this point. 

$\bullet$ Eq. \eqref{dP1 2D nu2}: consider behavior of $P_{n+1}$ under the blow-up \eqref{dP1 2D E2} at the point $[0:1:0]$. This point lies on the critical line $\{z=0\}$ but not on the critical line $\{y-z=0\}$. In the blow-up coordinate patch \eqref{dP1 2D E2} we have $z\sim u_2^2$ and $y-z\sim 1$, therefore
\begin{eqnarray*}
P_{n+1} & \sim & u_2^{-2\nu_3(n)}f^*P_n \; = \; u_2^{-2\nu_3(n)}P_n\big(1-u_2^2v_2, u_2^3v_2,(1-u_2^2v_2)^2\big)\\
& \sim & u_2^{-2\nu_3(n)}P_n(1, u_2^3v_2,1-u_2^2v_2)\; \sim \; u_2^{-2\nu_3(n)+\nu_5(n)}.
\end{eqnarray*}
The last step follows from the comparison of the argument of $P_n$ with the blow-up \eqref{dP1 2D E5}.

$\bullet$ Eq. \eqref{dP1 2D nu6}: consider behavior of $P_{n+1}$ under the blow-up \eqref{dP1 2D E6} at the point $[0:1:1]$. This point lies on the critical line $\{y-z=0\}$ but not on the critical line $\{z=0\}$. In the blow-up coordinate patch \eqref{dP1 2D E6} we have $y-z=u_6$ and $z=1$, therefore
\begin{eqnarray*}
P_{n+1}& = & u_6^{-\nu_2(n)}f^*P_n \; = \; u_6^{-\nu_4(n)}P_n(u_6(1+u_6), u_6v_6,u_6^2)\\
& = & u_6^{-\nu_2(n)+d(n)}P_n(1+u_6, v_6,u_6)\; \sim \; u_6^{d(n)-\nu_2(n)}.
\end{eqnarray*}
Indeed, $P_n(1,v_6,0)\not\equiv 0$ since, by construction, $P_n(x,y,z)$ is not divisible by $z$.

$\bullet$ Eq. \eqref{dP1 2D nu8}: consider behavior of $P_{n+1}$ under the blow-up \eqref{dP1 2D E8} at the point $[0:1:1]$. This point lies on the critical line $\{y-z=0\}$ but not on the critical line $\{z=0\}$. In the blow-up coordinate patch \eqref{dP1 2D E8} we have $y-z=u_8^2v_8$ and $z=1$, therefore
\begin{eqnarray*}
P_{n+1}& \sim & u_8^{-2\nu_2(n)}f^*P_n \; = \; u_8^{-2\nu_2(n)}P_n(u_8^2v_8(1+u_8^2v_8), u_8^3v_8,u_8^4v_8^2)\\
& \sim & u_8^{-2\nu_2(n)}u_8^{2d(n)}P_n(1, u_8,u_8^2v_8)\; \sim \; u_8^{2d(n)-2\nu_2(n)}u_8^{\nu_{10}(n)}.
\end{eqnarray*}
The last step follows from the comparison of the argument of $P_n$ with the blow-up \eqref{dP1 2D E10}.

$\bullet$ Eq. \eqref{dP1 2D nu10}: consider behavior of $P_{n+1}$ under the blow-up \eqref{dP1 2D E10} at the point $[1:0:0]$. This point lies on the both critical lines $\{z=0\}$ and $\{y-z=0\}$. In the blow-up coordinate patch \eqref{dP1 2D E10} we have $y-z\sim u_{10}$ and $z\sim u_{10}^2$. Therefore,
\begin{eqnarray*}
P_{n+1} & \sim & u_{10}^{-\nu_2(n)-2\nu_3(n)} f^*P_n \\
& = & u_{10}^{-\nu_2(n)-2\nu_3(n)}P_n\big(u_{10}^2(1-u_{10}v_{10}), u_{10}^2v_{10},u_{10}^2(1-u_{10}v_{10})^2\big)\\
& \sim & u_{10}^{-\nu_2(n)-2\nu_3(n)}u_{10}^{2d(n)}P_n(1, v_{10},1-u_{10}v_{10}) \; \sim \; u_{10}^{2d(n)-\nu_2(n)-2\nu_3(n)}.
\end{eqnarray*}
This finishes the proof.
\end{proof}

Theorem \ref{theorem Ex2 recurrent nu}  is a concrete realization of results of Section \ref{sect propagate} concerning propagation of local indices. More precisely,  formulas \eqref{dP1 2D nu3}, \eqref{dP1 2D nu5} exemplify \eqref{propagation regular}, formula \eqref{dP1 2D nu2} illustrates the case \eqref{propagation critical}, formulas \eqref{dP1 2D nu6}, \eqref{dP1 2D nu10} illustrate the case \eqref{propagation inverse critical}, while formula \eqref{dP1 2D nu6} is an instance of \eqref{propagation singular}. 

\begin{cor}
Degrees $d(n)$ satisfy the following recurrent relation:
\begin{equation}\label{dP1 2D d recur}
d(n+1)-2d(n)+d(n-2)+d(n-3)-2d(n-5)+d(n-6)=0.
\end{equation}
\end{cor}
\begin{proof}
Comparing \eqref{dP1 2D d} with \eqref{dP1 2D nu6}, we find:
$$
d(n+1)=d(n)+\nu_6(n+1)-\nu_3(n)\quad \Leftrightarrow \quad \nu_3(n+2)-\nu_3(n)=d(n+1)-d(n).
$$
Analogously, comparing \eqref{dP1 2D d} with \eqref{dP1 2D nu10}, we find:
$$
\nu_{10}(n+1)=d(n+1)-\nu_3(n) \quad \Leftrightarrow \quad \nu_2(n+4)=d(n+1)-\nu_3(n).
$$
From the last two equations we derive:
$$
\nu_2(n+4)-\nu_2(n+2)=d(n+1)-\nu_3(n)-d(n-1)+\nu_3(n-2)=d(n+1)-2d(n-1)+d(n-2).
$$
Now:
\begin{align*}
& d(n+1)-d(n-1) \\
& \quad =  2\big(d(n)-d(n-2)\big)-\big(\nu_3(n)-\nu_3(n-2)\big)-\big(\nu_2(n)-\nu_2(n-2)\big)\\
& \quad = 2\big(d(n)-d(n-2)\big)-\big(d(n-1)-d(n-2)\big)-\big(d(n-3)-2d(n-5)+d(n-6)\big),
\end{align*}
which is equivalent to \eqref{dP1 2D d recur}.
\end{proof}

If $P_0$ is a generic linear polynomial, for which $d(0)=1$ and all $\nu_i(0)=0$, then recurrent relations \eqref{dP1 2D d}--\eqref{dP1 2D nu10} easily produce the following table of degrees and indices:
\smallskip

\begin{center}
\begin{tabular}[h]{c|c|cccccc}
\hline
$n$ & $d(n)$ & $\nu_2(n)$ & $\nu_3(n)$ & $\nu_6(n)$  & $\nu_5(n)$  & $\nu_8(n)$ & $\nu_{10}(n)$ \\
\hline
0 & 1 & 0 & 0 & 0  & 0 & 0 & 0\\
1 & 2 & 0 & 0 & 1 & 0 & 2 & 2 \\
2 & 4 & 0 & 1 & 2 & 2 & 6 & 4\\
3 & 7 & 0 & 2 & 4 & 6 & 12 & 6 \\
4 & 12 & 2 & 4 & 7 & 12 & 20 & 10 \\
5 & 18 & 4 & 7 & 10 & 20 &  30 & 14 \\
6 & 25 & 6 & 10 & 14 & 30 & 42 & 18 \\
7 & 34 & 10 & 14 & 19 & 42 & 56 & 24\\
8 & 44 & 14 & 19 & 24 & 56 & 72 & 30\\
9 & 55 & 18 & 24 & 30 & 72 & 90 & 36\\
10 & 68 & 24 & 30 & 37 & 90 & 110 & 44\\
11& 82 & 30 & 37 & 44 & 110 & 132 & 52\\
12 & 97 & 36 & 44 & 52 & 132 & 156 & 60\\
13 & 114 & 44 & 52 & 61 & 156 & 182 & 70
\end{tabular}
\end{center}
\smallskip

By induction one proves in this case:
\begin{equation}\label{dP1 2D d expl}
d(n)=\frac{1}{9} \left(6 n^2-2\cos\frac{2\pi  n}{3}+11\right).
\end{equation}  
This is the formula for degrees $\deg(f^n)$.

\subsection{Degree computation based on the Picard group}

On the surface $X$ obtained from $\mbP^2$ by ten successive blow-ups, one can easily check that $f$ is lifted to an automorphism $f_X$. 
The action of $(f_X)_*$ on the divisors obtained by regularizing the singular orbits \eqref{ex2-singconf-1}, \eqref{ex2-singconf-2},  can be summarized as follows:
\begin{align}
  & \overline{\{y-z=0\}} \;\; \rightarrow \;\; \cE_2 \;\; \rightarrow \;\; \cE_5 \;\; \rightarrow \;\; \cE_8 \;\; \rightarrow \;\; \cE_{10} \;\; \rightarrow \;\; \overline{\{x-z=0\}},   \label{2d-Ex2-2}\\
  & \overline{\{z=0\}} \;\; \rightarrow \;\; \cE_3-\cE_4-\cE_5 \;\; \rightarrow \;\; \cE_6-\cE_7-\cE_8 \;\; \rightarrow \;\; \overline{\{z=0\}}, \label{2d-Ex2-1} \\
  & \cE_1-\cE_2 \;\; \rightarrow \;\; \cE_4-\cE_5 \;\;\rightarrow \;\; \cE_7-\cE_8 \;\; \rightarrow \; \;\cE_9-\cE_{10} \;\; \rightarrow \;\; \cE_1-\cE_2.
   \label{2d-Ex2-3}
\end{align}

The following features of Fig.\ \ref{fig:Divisors} are essential for determining the divisor classes on the left-hand side and on the right-hand side of the first two patterns:
\begin{itemize}
\item Divisor $E_1$ intersects the (proper transform of) the line $\{z=0\}$ at $v_1=0$; this is exactly the point which is blown up to $E_2$;
\item Divisor $E_4$ intersects the (proper transform of) the divisor $E_3$ at $v_4=0$; this is exactly the point which is blown up to $E_5$;
\item Divisor $E_7$ intersects the (proper transform of) the divisor $E_6$ at $v_7=0$; this is exactly the point which is blown up to $E_8$;
\item Divisor $E_9$ intersects the (proper transform of) the line $\{z=0\}$ at $v_9=0$; this is exactly the point which is blown up to $E_{10}$.
\end{itemize}
From this, we can read off that the divisor class for $\overline{\{z=0\}}$ is $H-\cE_9-\cE_{10}-\cE_1-\cE_2$,
and the divisor classes for $\overline{\{y-z=0\}}$ and $\overline{\{x-z=0\}}$ are $H-\cE_6-\cE_9$ and $H-\cE_1-\cE_3$, respectively.

From this, we obtain the action of the lifted map $(f_X)_*$ on $\Pic(X)$:
\begin{equation}\label{Ex2 Picard action}
(f_X)_*: \;  \left( \begin{array}{c}
H  \\
\cE_1\\
\cE_2 \\ 
\cE_3 \\ 
\cE_4\\ 
\cE_5 \\
\cE_6 \\
\cE_7 \\ 
\cE_8 \\
\cE_9 \\
\cE_{10}\\
\end{array} \right)
\rightarrow
\left( \begin{array}{c}
2H-\cE_1-\cE_2-\cE_3\\
\cE_4\\
\cE_5\\
\cE_6\\
\cE_7\\
\cE_8\\
H-\cE_1-\cE_2\\
\cE_9\\
\cE_{10}\\
H-\cE_2-\cE_3\\
H-\cE_1-\cE_3\\
\end{array} \right)
\end{equation}
As explained in Section \ref{sect method Picard}, the degrees $\deg(f^n)$ are given by the $(H,H)$-entry of the $n$-th power of the matrix of this map. Alternatively, upon setting \eqref{div P}, we have:
\begin{theo}\label{theorem Ex2 recurrent mu} 
The quantities $d(n)=\deg(P_n)$, $\mu_i(n)=\mu_i(P_n)$ satisfy the following recurrent relations:
\begin{eqnarray}   
  d(n+1) & = & 2d(n)-\mu_1(n)-\mu_2(n)-\mu_3(n),\\
 \mu_1(n+1) & = & \mu_4(n),\\
 \mu_2(n+1) & = & \mu_5(n),\\
 \mu_3(n+1) & = & \mu_6(n),\\
 \mu_4(n+1) & = & \mu_7(n),\\
 \mu_5(n+1) & = & \mu_8(n),\\
 \mu_6(n+1) & = & d(n)-\mu_1(n)-\mu_2(n),\\
 \mu_7(n+1) & = & \mu_9(n),\\
 \mu_8(n+1) & = & \mu_{10}(n),\\
 \mu_9(n+1) & = & d(n)-\mu_2(n)-\mu_3(n),\\
 \mu_{10}(n+1) & = & d(n)-\mu_1(n)-\mu_3(n).
\end{eqnarray}
\end{theo}
\begin{proof}
This follows from Corollary \ref{cor system d mu} and Remark \ref{remark automorphism}, since $f_X$ in the present case is an automorphism. Recurrent equations are directly read off the action of $(f_X)_*$ in \eqref{Ex2 Picard action}.
\end{proof}

\subsection{Relation between two methods}

For any homogeneous polynomial $P(x,y,z)$ we now introduce ten local indices $\mu_i(P)$, $i=1,\ldots,10$, corresponding to ten sigma-processes used to blow up $\mbP^2$ to $X$.

At $p_1=[0:1:0]$:
 \begin{eqnarray}
    P(u_1,1,u_1v_1) & = & u_1^{\mu_1(P)}\bar P(u_1,v_1), \quad \bar P(0,v_1)\not \equiv 0,\label{local index 1} \\
    P(u_2,1,u_2^2v_2) & = & u_2^{\mu_1(P)}u_2^{\mu_2(P)}\widetilde P(u_2,v_2),  \quad \widetilde P(0,v_2)\not \equiv 0. \label{local index 2}
 \end{eqnarray}
At $p_2=[1:0:1]$:
 \begin{eqnarray}
    P(1,u_3v_3,1-u_3) & = & u_3^{\mu_3(P)}\bar P(u_3,v_3), \quad \bar P(0,v_3)\not \equiv 0,  \label{local index 3}\\
    P(1,u_4^2v_4,1-u_4v_4) & = & (u_4v_4)^{\mu_3(P)}u_4^{\mu_4(P)}\hat P(u_4,v_4), \quad \hat P(0,v_4)\not \equiv 0, \label{local index 4}\\
    P(1,u_5^3v_5,1-u_5^2v_5) & = & (u_5^2v_5)^{\mu_3(P)}u_5^{\mu_4(P)}u_5^{\mu_5(P)}\widetilde P(u_5,v_5), \quad \widetilde P(0,v_5)\not \equiv 0. 
    \qquad \label{local index 5}
 \end{eqnarray}
 At $p_3=[0:1:1]$:
 \begin{eqnarray}
    P(u_6v_6,1+u_6,1) & = & u_6^{\mu_6(P)}\bar P(u_6,v_6), \quad \bar P(0,v_6)\not \equiv 0,  \label{local index 6}\\
    P(u_7^2v_7,1+u_7v_7,1) & = & (u_7v_7)^{\mu_6(P)}u_7^{\mu_7(P)}\hat P(u_7,v_7), \quad \hat P(0,v_7)\not \equiv 0, \label{local index 7}\\
    P(u_8^3v_8,1+u_8^2v_8,1) & = & (u_8^2v_8)^{\mu_6(P)}u_8^{\mu_7(P)}u_8^{\mu_8(P)}\widetilde P(u_8,v_8), \quad \widetilde P(0,v_8)\not \equiv 0. \qquad \label{local index 8}
 \end{eqnarray}
 At $p_4=[1:0:0]$:
 \begin{eqnarray}
    P(1,u_9,u_9v_9) & = & u_9^{\mu_9(P)}\bar P(u_9,v_9), \quad \bar P(0,v_9)\not \equiv 0,\label{local index 9} \\
    P(1,u_{10},u_{10}^2v_{10}) & = & u_{10}^{\mu_9(P)}u_{10}^{\mu_{10}(P)}\widetilde P(u_{10},v_{10}),  \quad \widetilde P(0,v_{10})\not \equiv 0. \label{local index 10}
 \end{eqnarray}
 In each group of equations, the arguments of $P$ for the next one are obtained from those from the previous one by a sigma-process, and the right-hand side gains in this transition an extra factor corresponding to the index of $P$ with respect to this sigma-process. 
 
 We remark that relations from Theorem \ref{theorem Ex2 recurrent mu} can be alternatively proven by a direct analysis similar to that used for the proof of Theorem \ref{theorem Ex2 recurrent nu}. 

\begin{pro} \label{prop Ex2 nu from mu}
Given a homogeneous polynomial $P$, the indices $\nu_2(P)$, $\nu_5(P)$, $\nu_8(P)$ and $\nu_{10}(P)$ introduced in Definition \ref{definition indices nu}, are related to the indices $\mu_i(P)$ via
\begin{eqnarray}
    \nu_2(P) & = & \mu_1(P)+\mu_2(P), \label{nu2 thru mu} \\
    \nu_5(P) & = & 2\mu_3(P)+\mu_4(P)+\mu_5(P), \label{nu5 thru mu} \\
    \nu_8(P) & = & 2\mu_6(P)+\mu_7(P)+\mu_8(P), \label{nu8 thru mu} \\
    \nu_{10}(P) & = & \mu_9(P)+\mu_{10}(P). \label{nu10 thru mu}
\end{eqnarray}
With these relations, and upon setting $\nu_3(P)=\mu_3(P)$ and $\nu_6(P)=\mu_6(P)$, recurrent relations of Theorem \ref{theorem Ex2 recurrent nu} follow from those of Theorem \ref{theorem Ex2 recurrent mu}.
\end{pro}
\begin{proof} The first statement follows directly from \eqref{local index 2}, \eqref{local index 5}, \eqref{local index 8}, and \eqref{local index 10}, respectively. The second is a result of a straightforward check.
\end{proof}

\subsection{Finding an integral of motion}

To find the integral of motion, we again compute the eigenvectors corresponding to the eigenvalue 1. There are three independent such vectors:
$$
H-\cE_1-\cE_5-\cE_8-\cE_{10}, \quad H-\cE_6-\cE_4-\cE_7-\cE_9, \quad H-\cE_3-\cE_2.
$$
With a linear combination of these three vectors, we find the divisor class with the self-intersection number 0,
$$
4H-2\cE_3-2\cE_2-\cE_6-\cE_1-\cE_4-\cE_5-\cE_7-\cE_8-\cE_9-\cE_{10},
$$
whose linear system has dimension 1, and is of the form
$$
-hx^2y^2+x^3z+3x^2yz+3xy^2z+y^3z-3x^2z^2-6xyz^2-3y^2z^2+3xz^3+3yz^3-z^4.
$$
Back to inhomogeneous coordinates, this means we have the following candidate for a rational integral of motion:
$$
h(x,y)=\frac{(x+y-1)^3}{x^2y^2},
$$
which can be now verified by direct computation.


\section{Example 3: a linearizable 2D system}
\label{sect linearizable}

The example in this section illustrates the situation when the algebraically stable modification of a birational map of $\mbP^2$ is \emph{not} an automorphism. A straightforward application of our algorithm would require an infinite number of blow-ups, but actually the algorithm can be made finite by a more careful analysis. This corresponds to the existence of an algebraically stable modification via a finite number of blow-ups. The system is actually linearizable, and the situation closely resembles that of \cite{Takenawa_et-al_2003}. 

\subsection{Definition of the map}

The following second order difference equation was considered in \cite{Ablowitz_Halburd_Herbst}: 
$$
\xi_{n+1} = \xi_n + \dfrac{\xi_n - \xi_{n-1}}{1+(\xi_n - \xi_{n-1})}.
$$ 
It can be rewritten as a 2D discrete system 
$$ 
(\xi_{n+1},\eta_{n+1}) = \left(\xi_n + \dfrac{\xi_n - \eta_n}{1+(\xi_n - \eta_n)}, \xi_n \right), 
$$ 
and finally as a birational map in homogeneous coordinates
\begin{equation}\label{A f}
f \begin{bmatrix}
x \\ y \\ z
\end{bmatrix} = 
\begin{bmatrix}
x (x-y+z)+(x-y)z \\ x(x-y+z) \\ z(x-y+z)
\end{bmatrix}, \qquad 
f^{-1}
\begin{bmatrix}
x \\ y \\ z
\end{bmatrix}= 
\begin{bmatrix}
y(y-x+z) \\ y(y-x+z)+(y-x)z  \\ z(y-x+z)
\end{bmatrix}.
\end{equation}

The singularities of the map $f$ and of the inverse map $f^{-1}$ are
$$
{\mathcal I}(f)=\{p_2=[1:1:0],\; p_3=[0:1:0]\}
$$ 
and 
$$
{\mathcal I}(f^{-1})=\{p_1=[1:0:0], \; p_2=[1:1:0]\}.
$$
The respective critical sets are:
$$
{\mathcal C}(f)=\{z=0\}\cup\{x-y+z=0\}
$$ 
and 
$$
{\mathcal C}(f^{-1})=\{z=0\}\cup\{y-x+z=0\}.
$$
Both components of ${\mathcal C}(f)$ are degree lowering curves:
\begin{equation}\label{ex3-singconf-1}
    f: \quad \{z=0\} \;\rightarrow \; p_2   \; \in \; {\mathcal I}(f).
\end{equation}
and
\begin{equation}\label{ex3-singconf-2}
    f: \quad \{x-y+z=0\} \;\rightarrow \; p_1 \;\rightarrow \; p_2  \; \in \; {\mathcal I}(f).
\end{equation}

\subsection{Regularizing the map $f$}
\label{ss:regf}

{\bf Action of $f$ on the critical line $\{x-y+z=0\}$.} This line is mapped by $f$ to the point $p_1=[1:0:0]$. In order to resolve this, we  blow up this point. We use the following change to coordinates $(u_1,v_1)$ in the affine chart $x=1$:
\begin{equation}\label{A E1}
\phi_1=\pi_1:\; \left\{\begin{array}{l} x=1, \\ y=u_1, \\ z=u_1v_1.\end{array}\right.
\end{equation}
One can interpret elements $(0,v_1)\in E_1$ as the germs $\{x=1, \; z=v_1y\}$ over $p_1$.
In coordinates \eqref{A E1} for the image $[\widetilde x:\widetilde y:\widetilde z]$, we compute:
$$
\widetilde u_1=\frac{\widetilde{y}}{\widetilde{x}}=\frac{x(x-y+z)}{x(x-y+z)+(x-y)z}, \quad \widetilde v_1=\frac{\widetilde z}{\widetilde y}=\frac{z}{x}.
$$
As soon as $x-y+z=0$, we have $\widetilde u_1=0$ and $\widetilde v_1=z/x$. Thus, (the proper transform of) the line $\{x-y+z=0\}$ is biregularly mapped to the exceptional divisor $E_1$:
$$
f|_{\{x-y+z=0\}}:[x:x+z:z]\mapsto (0,\widetilde v_1)=(0,z/x)\in E_1.
$$
\smallskip

{\bf Action of $f$ on the critical line $\{z=0\}$.}  This line is mapped by $f$ to the point $p_2=[1:1:0]$. In order to resolve this, we  blow up this point. We use the following change to coordinates $(u_2,v_2)$ in the affine chart $x=1$:
\begin{equation}\label{A E2}
\phi_2=\pi_2:\; \left\{\begin{array}{l} x=1+u_2v_2, \\ y=1, \\ z=u_2.\end{array}\right.
\end{equation}
The exceptional divisor is given by $E_2=\{u_2=0\}$.
In coordinates \eqref{A E2} for the image $[\widetilde x:\widetilde y:\widetilde z]$, we compute:
$$
\widetilde u_2=\frac{\widetilde{z}}{\widetilde{y}}=\frac{z(x-y+z)}{x (x-y+z)}=\frac{z}{x}, \quad \widetilde v_2=\frac{{\widetilde x}-{\widetilde{y}}}{\widetilde{z}}=\frac{(x-y)z}{z(x-y+z)}=\frac{x-y}{x-y+z}.
$$ 
Thus, the image by $f$ of the line $\{z=0\}$ is the single point $(\widetilde u_2,\widetilde v_2)=(0,1)$ in $E_2$. In order to resolve this, one more blow-up is necessary.  For this second blow-up, we use the following change to coordinates $(u_3,v_3)$:
\begin{equation}\label{A E3}
\pi_3:\; \left\{\begin{array}{l} u_2=u_3, \\ v_2=1+u_3v_3\end{array}\right. \quad \Rightarrow\quad 
\phi_3=\pi_2\circ\pi_3:\; \left\{\begin{array}{l} x=1+u_3+u_3^2v_3, \\ y=1, \\ z=u_3. \end{array}\right.
\end{equation}
One can interpret elements $(0,v_3)\in E_3$ as the germs $\{y=1, \; x=1+z+v_3z^2\}$ over $p_2$.
In coordinates \eqref{A E3} for the image $[\widetilde x:\widetilde y:\widetilde z]$, we compute (at $z=0$):
$$
\widetilde u_3=\frac{\widetilde{z}}{\widetilde{y}}=\frac{z}{x}, \quad \widetilde v_3=\frac{(\widetilde x-\widetilde y-\widetilde z)\widetilde y}{\widetilde z^2}=-\frac{x}{x-y+z}.
$$
Thus, (the proper transform of) the line $\{z=0\}$ is biregularly mapped to the exceptional divisor $E_3$:
$$
f|_{\{z=0\}}:[x:y:0] \mapsto (0,\widetilde v_3)=\big(0,-x/(x-y)\big)\in E_3.
$$
\smallskip

{\bf Action of $f$ on $E_2$.} It appears useful to consider next the action of $f$ on $E_2$, prior to considering the propagation of indices.  Recall that we parametrize a neighborhood of $E_2$ by \eqref{A E3}. We compute (if $v_2\neq -1$):
\begin{eqnarray*}
f\left[\begin{array}{c} 1+u_2v_2\\ 1 \\ u_2\end{array}\right] & = & \left[\begin{array}{c} (1+u_2v_2)u_2(1+v_2)+u_2^2v_2 \\ (1+u_2v_2)u_2(1+v_2) \\ u_2^2(1+v_2) \end{array}\right]=\left[\begin{array}{c}  1+\dfrac{u_2v_2}{(1+u_2v_2)(1+v_2)} \\ 1 \\  \dfrac{u_2}{1+u_2v_2}  \end{array}\right]\\
 & = & \left[\begin{array}{c} 1+ \widetilde{u}_2\widetilde{v}_2 \\  1 \\ \widetilde{u}_2 \end{array}\right],
\end{eqnarray*}
where
$$
\widetilde{u}_2=\dfrac{u_2}{1+u_2v_2},   \quad  \widetilde{v}_2=\dfrac{v_2}{1+v_2}.
$$
We see that $f$ maps $E_2\setminus \{(0,-1)\}$ to itself, with $\widetilde{v}_2=v_2/(1+v_2)$. The point  $(0,-1)\in E_2$ is a singular point of $f$; the exceptional divisor of the further blow-up at this point is biregularly mapped to the line $\{z=0\}$.

{\bf Action of $f$ on $E_1$.} Recall that we parametrize a neighborhood of $E_1$ by \eqref{A E1}. We compute:
\begin{eqnarray*}
f\left[\begin{array}{c} 1\\ u_1 \\ u_1v_1\end{array}\right] & = & \left[\begin{array}{c} 1-u_1+u_1v_1+u_1v_1(1-u_1)\\ 1-u_1+u_1v_1 \\ u_1v_1(1-u_1+u_1v_1) \end{array}\right]=\left[\begin{array}{c}  1+\dfrac{u_1v_1(1-u_1)}{1-u_1+u_1v_1} \\ 1 \\  u_1v_1  \end{array}\right]=:\left[\begin{array}{c}  \widetilde{x} \\  \widetilde{y} \\ \widetilde{z} \end{array}\right].
\end{eqnarray*}
In coordinates \eqref{A E2}, we have: 
$$
\widetilde{u}_2=u_1v_1,   \quad  \widetilde{v}_2=\dfrac{1-u_1}{1-u_1+u_1v_1}=1-\frac{u_1v_1}{1-u_1+u_1v_1}=1-\widetilde{u}_2-\widetilde{u}_2^2\frac{1-v_1}{v_1}+O(\widetilde{u}_2^3).
$$
We see that $f$ maps $E_1$ to a single point $(\widetilde{u}_2,\widetilde{v}_2)=(0,1)\in E_2$.  According to the result above on the action of $f$ on $E_2$, we can conclude that the further iterates $f^m(E_1)$ are definite points $(0,v_2^{(m)})\in E_2$ with $v_2^{(m+1)}=v_2^{(m)}/(1+v_2^{(m)})$, so that
$v_2^{(m)}=1/m$. In other words, on the surface obtained from $\mbP^2$ by blow-ups at $p_1$ and at $p_2$, the line $\{x-y+z=0\}$ is not a degree lowering curve.  The same is true for the line $\{z=0\}$. The lift of $f$ to this surface is algebraically stable.

For the algorithm for computing $\deg(f^n)$ via local indices of polynomials, we would need to propagate by $f$ the blow-ups $\phi_1$ at $p_1$ and $\phi_3$ at $p_2$. To resolve the contraction of $E_1$ to the point $(0,1)\in E_2$, further blow-ups at this point would be necessary: first, according to \eqref{A E3}, where we find
$$
\widetilde{u}_3=u_1v_1,   \quad  \widetilde{v}_3=-1+O(\widetilde u_3).
$$
so that $f$ maps $E_1$ a single point $(\widetilde{u}_3,\widetilde{v}_3)=(0,-1)\in E_3$. A further blow-up at this point:
\begin{equation}\label{A E4}
\pi_4:\;\left\{\begin{array}{l} u_3=u_4, \\ v_3=-1+u_4v_4\end{array}\right. \quad \Rightarrow\quad 
\phi_4=\pi_2\circ \pi_3\circ\pi_4:\; \left\{\begin{array}{l} x=1+u_4-u_4^2+u_4^3v_4, \\ y=1, \\ z=u_4, \end{array}\right.
\end{equation}
will give a regular map 
$$
f|_{E_1}:{E_1}\ni (0,v_1)\mapsto (0,\widetilde v_4)=\big(0,-(1-v_1)/v_1\big)\in E_4.
$$
The further propagation steps would lead to blow-up at the points $(0,v_2^{(2)})=(0,1/2)\in E_2$, $(0,v_2^{(3)})=(0,1/3)\in E_2$, etc. Thus, the full regularization of the orbit of the line $\{x-y+z=0\}$ would require an infinite sequence of blow-ups. The same is true for the orbit of the line $\{z=0\}$. Compare this situation with \cite{Takenawa_et-al_2003}. However, as we will see below, one does not need indices other than $\nu_1=\nu_{\phi_1}$, $\nu_2=\nu_{\phi_2}$ and $\nu_3=\nu_{\phi_3}$ to obtain a finite system of recurrent relations.

\subsection{Degree computation based on local indices}

\begin{de} Let $P(x,y,z)$ be a generic homogeneous polynomial. 
\begin{itemize}
	\item We define the index $\nu_1(P)=\nu_{\phi_1}(P)$ by 
$$
P\circ\phi_1=P(1,u_1,u_1v_1) = u_1^{\nu_1(P)} \widetilde{P}(u_1,v_1),
$$ 
where $\widetilde P(0,v_1) \not\equiv 0$.
\item We define the index $\nu_2(P)=\nu_{\phi_2}(P)$ by 
$$
P\circ\phi_2=P(1+u_2v_2, 1, u_2) = u_2^{\nu_2(P)} \widetilde{P}(u_2,v_2),
$$
where $\widetilde{P}(0,v_2) \not\equiv 0$.
	\item We define the index $\nu_3(P)=\nu_{\phi_3}(P)$ by 
$$
P\circ \phi_3=P(1+u_3+u_3^2v_3, 1, u_3) = u_3^{\nu_3(P)} \widetilde{P}(u_3,v_3),
$$
where $\widetilde{P}(0,v_3) \not\equiv 0$.
\end{itemize}
\end{de}

We have the following result (a concrete realization of Theorem \ref{theorem split off} in the present case).

\begin{pro}\label{lemma Example 3 pullback}
For a homogeneous polynomial $P=P(x,y,z)$, its pullback $f^*P = P \circ \tilde f$ is divisible by $(x-y+z)^{\nu_1(P)} z^{\nu_3(P)}$ (and by no higher degrees of $x-y+z$, resp. of $z$).
\end{pro}
\begin{proof}
\begin{enumerate}[1)]
\item In the chart \eqref{A E1}, we compute:
$$
P(1,y,z)=u_1^{\nu_1}\widetilde P(u_1,v_1)=y^{\nu_1}\widetilde P\Big(y,\frac{z}{y}\Big).
$$
In homogeneous coordinates,
$$
P(x,y,z)=x^d P\Big(1,\frac{y}{x},\frac{z}{x}\Big)=y^{\nu_1}x^{d-\nu_1}\widetilde P\Big(\frac{y}{x},\frac{z}{y}\Big).
$$
Setting $\tilde f(x,y,z)=(X,Y,Z)$, we compute, according to \eqref{A f}:
$$
f^*P(x,y,z)=Y^{\nu_1}X^{d-\nu_1}\widetilde P\Big(\frac{Y}{X},\frac{Z}{Y}\Big)=\big(x(x-y+z)\big)^{\nu_1}X^{d-\nu_1}\widetilde P\Big(\frac{x(x-y+z)}{X},\frac{z}{x}\Big).
$$
This polynomial is manifestly divisible by $(x-y+z)^{\nu_1}$ (and by no higher degree of $x-y+z$). 
\item Likewise, we compute in the chart \eqref{A E3}:
$$
P(x,1,z)=u_3^{\nu_3}\widetilde P(u_3,v_3)=z^{\nu_3}\widetilde P\Big(z,\frac{x-1-z}{z^2}\Big).
$$
In homogeneous coordinates,
$$
P(x,y,z)=y^d P\Big(\frac{x}{y},1,\frac{z}{y}\Big)=z^{\nu_3}y^{d-\nu_3}\widetilde P\Big(\frac{z}{y},\frac{(x-y-z)y}{z^2}\Big).
$$
Therefore, upon a straightforward computation according to \eqref{A f}:
$$
\begin{aligned}
f^*P(x,y,z)&=Z^{\nu_3}Y^{d-\nu_3}\widetilde P\Big(\frac{Z}{Y},\frac{(X-Y-Z)Y}{Z^2}\Big)\\&=z^{\nu_3}x^{d-\nu_3}(x-y+z)^d\widetilde P\Big(\frac{z}{x},-\frac{x}{x-y+z}\Big).
\end{aligned}
$$
This polynomial is manifestly divisible by $z^{\nu_3}$ (and by no higher degree of $z$). 
\end{enumerate}
\end{proof}

\begin{lemma} \label{Ex3 lemma pullback}
The indices $\nu_1, \nu_2, \nu_3$ satisfy the following relations:
\begin{eqnarray} 
\nu_1 (f^*P) & = &  \nu_3(P), \label{A pullback nu1}\\
\nu_2(f^*P) & =  & \deg(P) + \nu_2(P), \label{A pullback nu2}\\
\nu_3(f^*P) & =  & \deg(P) + \nu_2(P). \label{A pullback nu3}
\end{eqnarray}
\end{lemma}
\begin{proof}
Eq. \eqref{A pullback nu1} follows from the fact that the image by $f$ of the blow-up \eqref{A E1} is given by \eqref{A E4} which manifestly belongs to \eqref{A E3}. For \eqref{A pullback nu2}, we refer to the end of Section 1: the image by $f$ of $E_2$ is $E_2$, and all three components of $\tilde{f}\circ \phi_2$ are divisible by $u_2$. Finally, \eqref{A pullback nu3} follows in the same way from the fact that the image by $f$ of $E_3$ is still in $E_2$ (with $\widetilde{v}_2=1/2$).
\end{proof}

\begin{re}
We see that for any homogeneous polynomial $P$, the indices $\nu_2$ and $\nu_3$ for $f^*P$ are equal. Of course, these indices must not coincide for an arbitrary polynomial. A simple counterexample:  $\nu_2(x-y-z)=1$ but $\nu_3(x-y-z)=2$.
\end{re}

For any polynomial $P=P_0$, we define iteratively
$$ 
P_{n+1}=(x-y+z)^{-\nu_1(P)} z^{-\nu_3(P)}(f^*P_n).
$$ 
For the degrees, there follows the relation
\begin{equation}\label{Ex3 d n+1}
\deg(P_{n+1}) = 2\deg(P_n) - \nu_1 (P_n) - \nu_3 (P_n).
\end{equation}
For the indices we have:
\begin{lemma} \label{Ex3 lemma recur}
The indices $\nu_1,\nu_2,\nu_3$ satisfy the following relations:
\begin{eqnarray}
 \nu_1 (P_{n+1}) & = & 0,  \label{Ex3 nu1  n+1}\\
\nu_2 (P_{n+1}) & =  & \deg(P_n)-\nu_1(P_n)+\nu_2(P_n)-\nu_3(P_n), \label{Ex3 nu2  n+1} \\
\nu_3(P_{n+1}) & = & \deg(P_n)-\nu_1(P_n)+\nu_2(P_n)-\nu_3(P_n). \label{Ex3 nu3  n+1}
\end{eqnarray}
\end{lemma} 
\begin{proof}
We have:
$$
\nu_1(P_{n+1})=\nu_1(f^*P_n)-\nu_1(P_n)\nu_1(x-y+z)-\nu_3(P_n)\nu_1(z)=\nu_3(P_n)-\nu_3(P_n)=0,
$$
since $\nu_1(x-y+z)=0$ and $\nu_1(z)=1$. 

Analogously,
\begin{eqnarray*}
\nu_2(P_{n+1}) & = & \nu_2(f^*P_n)-\nu_1(P_n)\nu_2(x-y+z)-\nu_3(P_n)\nu_2(z)\\
 & = & \deg(P_n)+\nu_2(P_n)-\nu_1(P_n)-\nu_3(P_n).
\end{eqnarray*}
since $\nu_2(x-y+z)=1$ and $\nu_2(z)=1$.

Finally,
\begin{eqnarray*}
\nu_3(P_{n+1}) & = & \nu_3(f^*P_n)-\nu_1(P_n)\nu_3(x-y+z)-\nu_3(P_n)\nu_3(z)\\
 & = & \deg(P_n)+\nu_2(P_n)-\nu_1(P_n)-\nu_3(P_n).
\end{eqnarray*}
since $\nu_3(x-y+z)=1$ and $\nu_3(z)=1$.
\end{proof}

\begin{theo}
For any polynomial $P$, the degrees of the polynomials $P_n$ are given by
\begin{equation}\label{Ex3 d expl}
\deg(P_n)=n\deg(P_1)-(n-1)\deg(P_0).
\end{equation}
For a generic linear polynomial $P_0$ (with all $\nu_i(P_0)=0$), we have $\deg(P_0)=1$, $\deg(P_1)=2$, and $\deg(P_n)=n+1$.
\end{theo}
\begin{proof}
From Lemma \ref{Ex3 lemma recur} we see that, starting with $n=1$, we will have $\nu_1(P_n)=0$, $\nu_2(P_n)=\nu_3(P_n)$, therefore, according to \eqref {Ex3 nu2  n+1}, \eqref{Ex3 nu3  n+1}: $\nu_2(P_{n+1})=\nu_3(P_{n+1})=\deg(P_n)$, and from \eqref{Ex3 d n+1}:
$$ 
\deg(P_{n+1}) = 2 \deg(P_n) - \deg(P_{n-1}), \quad n\ge 2.
$$
Now, \eqref{Ex3 d expl} follows by induction.
\end{proof}

\subsection{Degree computation based on the Picard group}

The algorithm from \cite{Birkett2022} for finding an algebraically stable modification of a given birational map, described in Section \ref{section first algorithm}, terminates in the present case after just one step. The only minimal destabilizing orbit of $f$ is $\{p_2\}$, and the lift of $f$ to the resulting blow-up surface $X=Bl_{p_2}(\mbP^2)$ is algebraically stable (although not an automorphism, because the proper transforms of the original critical lines are still contracted by $f_X$ to points). We have: $\Pic(X)=\mbZ H\oplus\mbZ \cE_2$, and one computes the following action of $f_X^*$ on $\Pic(X)$:
\begin{equation}\label{Ex3 Picard}
f_X^*: \begin{pmatrix} H \\ \cE_2 \end{pmatrix} \to \begin{pmatrix} 2H-\cE_2 \\ H \end{pmatrix}.
\end{equation}
Indeed, the pull-back of a generic line is a conic passing through $p_2$, while the pull-back of $\cE_2$ is a reducible divisor consisting of the proper transform of the line $\{z=0\}$, whose class is $H-\cE_2$, and of $\cE_2$ itself. Observe that $f_X^*$ does not preserve the intersection product.

Formula \eqref{Ex3 Picard} agrees with the recurrent formulas
\begin{align*}
\deg(P_{n+1}) & =2\deg(P_n)-\mu_2(P_n), \\
\mu_2(P_{n+1}) & =\deg(P_n),
\end{align*}
which hold for $n\ge 1$ (recall, in this case $\mu_2(P)=\nu_2(P)$). The powers of the matrix of $f_X^*$ are given by
$$
\begin{pmatrix} 2 & 1 \\ -1 & 0 \end{pmatrix}^n=\begin{pmatrix} n+1 & n \\ -n & 0 \end{pmatrix},
$$
which confirms that $\deg(f^n)=n+1$.


\section{Conclusion}

We have addressed the problem of computing the degrees of iterates of a birational map $f:\mbP^N\dasharrow\mbP^N$. To this end, we have developed a method  based on two main ingredients: 
\begin{itemize}
\item a detailed study of factorization of a polynomial under pull-back of $f$, based on local indices of a polynomial associated to blow-ups used to resolve the contraction of hypersurfaces by $f$, and 
\item a detailed study of propagation of these indices along orbits of $f$.
\end{itemize}
For maps admitting algebraically stable modifications $f_X:X\dasharrow X$, where $X$ is a variety obtained from $\mbP^N$ by a finite number of blow-ups, this method leads to an algorithm producing a finite system of recurrent equations relating the degrees and indices of iterated pull-backs of linear polynomials. This method is related, but not equivalent, to another method described in \cite{Viallet_2015,Viallet_2019}, based on the computation of the action of the pull-back of $f_X$ on the Picard group $\Pic(X)$. 

We have moreover illustrated the method by three representative two-dimensional examples. Since our method is applicable in any dimension, we will provide a number of three-dimensional examples an a separate companion paper \cite{Alonso_Suris_Wei_2023}. 
\medskip

\textbf{Acknowledgement.}
This research is supported by the DFG Collaborative Research Center TRR 109 ``Discretization in
Geometry and Dynamics''.

\bibliographystyle{siam}
\bibliography{Refs-2022}
\end{document}